\documentclass[12pt]{amsart}
\usepackage{amscd,amssymb,amsmath,amsthm,mathrsfs,lmodern,amsfonts,enumerate}

\usepackage[all]{xy}
\usepackage{comment}

\usepackage[utf8]{inputenc}
\usepackage[T1]{fontenc}
\usepackage[colorlinks=true,citecolor=PineGreen,filecolor=PineGreen,linkcolor=Mahogany,pagebackref, hyperindex]{hyperref}
\usepackage[dvipsnames]{xcolor}
\usepackage{color, hyperref}

\usepackage{fullpage, todonotes, amscd}

\usepackage{tikz}

\newtheorem{Theorem}{Theorem}[section]
\newtheorem{Theoremx}{Theorem}
 
\newtheorem{Potential Theorem}[Theorem]{Potential Theorem}
\newtheorem{Lemma}[Theorem]{Lemma}
\newtheorem{Corollary}[Theorem]{Corollary}
\newtheorem{Proposition}[Theorem]{Proposition}
\theoremstyle{definition}
\newtheorem{Example}[Theorem]{Example}
\newtheorem{Condition}[Theorem]{Condition}

\newtheorem{Definition}[Theorem]{Definition}
\newtheorem{Question}[Theorem]{Question}

\theoremstyle{remark} %we don't have to use this style thou
\newtheorem{Remark}[Theorem]{Remark}

\DeclareMathOperator{\sdim}{sdim}

\DeclareMathOperator{\height}{ht}

\DeclareMathOperator{\Max}{Max}

\DeclareMathOperator{\Supp}{Supp}
\DeclareMathOperator{\Ann}{Ann}

\DeclareMathOperator{\yyy}{div}

\DeclareMathOperator{\Hom}{Hom}
\DeclareMathOperator{\Spec}{Spec}
\DeclareMathOperator{\Min}{Min}

\DeclareMathOperator{\Tor}{Tor}

\DeclareMathOperator{\im}{im}

\DeclareMathOperator{\rank}{rank}
\DeclareMathOperator{\frk}{frk}

\DeclareMathOperator{\s}{s}
\DeclareMathOperator{\Assh}{Assh}
\DeclareMathOperator{\sr}{sr}
\DeclareMathOperator{\eHK}{e_{HK}}
\DeclareMathOperator{\ann}{ann}
\renewcommand{\leq}{\leqslant}
\renewcommand{\geq}{\geqslant}
\def\fsl{\operatorname{fs}} 
\def\QP{\mathcal{Q(P)}}

\def\c{\mathfrak{c}}
\def\PP{\mathcal{P}}

\def\p{\mathfrak{p}}
\def\q{\mathfrak{q}}
\def\m{\mathfrak{m}}

\def\a{\mathfrak{a}}
\def\b{\mathfrak{b}}
\def\Z{\mathbb{Z}}
\def\R{\mathbb{R}}
\def\Q{\mathbb{Q}}

\def\F{\mathbb{F}}

\def\D{\mathscr{D}}

\def\ds{\displaystyle}
\def\lra{\longrightarrow}
\def\div{\yyy}
\def\P{\mathcal{P}}
\renewcommand{\ge}{\geqslant}
\renewcommand{\le}{\leqslant}

\begin{document}

\title{Global Frobenius Betti numbers and F-splitting ratio}
\author{Alessandro De Stefani}
\address{Department of Mathematics, University of Nebraska, 203 Avery Hall, Lincoln, NE 68588}
\email{adestefani2@unl.edu}
\author{Thomas Polstra}
\address{Department of Mathematics, University of Utah, Salt Lake City, UT 84102 USA}
\email{polstra@math.utah.edu}
\author{Yongwei Yao}
\address{Department of Mathematics and Statistics, Georgia State University, Atlanta, Georgia 30303, USA}
\email{yyao@gsu.edu}

\maketitle

\begin{abstract} We extend the notion of Frobenius Betti numbers and F-splitting ratio to large classes of finitely generated modules over rings of prime characteristic, which are not assumed to be local. We also prove that the strong F-regularity of a pair $(R,\D)$, where $\D$ is a Cartier algebra, is equivalent to the positivity of the global F-signature $\s(R,\D)$ of the pair. This extends a result proved in \cite{DSPY}, by removing an extra assumption on the Cartier algebra.

 \end{abstract}

\section{Introduction}

Throughout this article, $R$ will denote a commutative Noetherian ring with unity, and of prime characteristic $p>0$. For $e \in \Z_{>0}$, let $F^e:R \to R$ be the $e$-th iterate of the Frobenius endomorphism, that is, the $p^e$-th power map on $R$. Moreover, let $F^e_*R$ denote $R$ as a module over itself, under restriction of scalars via $F^e$. Unless otherwise stated, we assume that $R$ is F-finite, that is, the Frobenius is a finite map. The goal of this article is to extend to the global setting numerical invariants related to Frobenius, that are typically only defined for local rings. This follows the path started in \cite{DSPY}, where these authors extended the definitions of Hilbert-Kunz multiplicity and F-signature to rings that are not necessarily local.

The first invariants that we consider in this article are the {\it Frobenius Betti numbers} $\beta_i^F(R)$ (see Sections \ref{Background} and \ref{Section betas} for definitions). When $(R,\m,k)$ is local, these invariants provide an asymptotic measure of the Betti numbers of the modules $F^e_*R$ (see \cite{DSHNB}):
\[
\ds \beta_i^F(R) = \lim_{e \to \infty} \frac{\lambda_R(\Tor_i^R(k,F^e_*R))}{\rank(F^e_*R)}.
\]
Here, $\lambda_R(-)$ denotes the length of an $R$-module. In the local case, it is known that the vanishing of Frobenius Betti numbers detects regularity of the ring \cite{AberbachLi}. To prove our results about Frobenius Betti numbers, we introduce certain auxiliary invariants, that we call {\it Frobenius Euler characteristics}, and we denote by $\chi_i^F(R)$. If $(R,\m,k)$ is local, these invariants are defined as
\[
\ds \chi_i^F(R) = \lim_{e \to \infty} \left[ \sum_{j=0}^i (-1)^{i-j} \frac{\lambda_R(\Tor_j^R(k,F^e_*R))}{\rank(F^e_*R)}\right].
\]

We define Frobenius Betti numbers and Frobenius Euler characteristic for rings that are not necessarily local, giving an appropriate interpretation of the numerators in the limits above as minimal number of generators of syzygies in appropriate resolutions, which we call minimal. We show existence of the resulting limits, and exhibit some relations with invariants coming from the localizations and with the singularities of $R$. Our first main result is the following:
\begin{Theoremx} [See Theorems \ref{global betas}, \ref{THM beta regular} and Proposition \ref{prop graded betti}] \label{THMx betas}
Let $R$ be an F-finite domain of prime characteristic $p>0$, not necessarily local. Then
\begin{enumerate}
\item The limits $\beta_i^F(R)$ and $\chi^F_i(R)$ exist.
\item \label{THMx betas 2} We have equality $\chi^F_i(R) = \max\{\chi_i^F(R_P)\mid P \in \Spec(R)\}$.
\item \label{THMx betas 3} $R$ is regular if and only if $\beta_i^F(R) = 0$ for some $i>0$.
\item If $R$ is a positively graded algebra over a local ring $(R_0,\m_0)$, then $\beta_i^F(R) = \beta_i^F(R_\m)$ and $\chi_i^F(R) = \chi_i^F(R_\m)$, where $\m = \m_0 + R_{>0}$.
\end{enumerate}
\end{Theoremx}
We currently do not know whether a relation, analogous to equality in part (\ref{THMx betas 2}) of Theorem~\ref{THMx betas}, also holds for $\beta_i^F(R)$. Similarly, we do not state any analogue of part (\ref{THMx betas 3}) for $\chi_i^F(R)$, since we are not aware of any interesting connections of these invariants with the regularity of $R$, even in the local case.

The second part of this article is focused on extending the {\it F-splitting ratio} to the global setting. The F-splitting ratio, denoted $r_F(R)$, is a measure of the asymptotic free-rank of the modules $F^e_*R$. It is defined similarly to the F-signature, and in certain cases it coincides with it. However, $r_F(R)$ is always positive for F-pure rings, while the F-signature is non-zero only for strongly F-regular rings. Its existence as a limit was first proved by Tucker for local rings \cite{Tucker2012}, while its positivity for F-pure local rings was established in \cite{BST2012}. We give a definition of this invariant for rings that are not necessarily local, by properly reinterpreting the notion of splitting dimension. We call this extension of the splitting dimension the {\it splitting rate of $R$}, and denote it by $\sr(R)$. Our main results on $\sr(R)$ and $r_F(R)$ can be summarized as follows:
\begin{Theoremx} [see Theorem \ref{Global F-splitting ratio exists} and Proposition \ref{prop graded r_F}] \label{THMx r_F}
Let $R$ be an F-finite domain of prime characteristic $p>0$. Then
\begin{enumerate}
\item The limit $r_F(R)$ exists.
\item We have equalities 
\[
\ds \sr(R) = \min\{\sr(R_P) \mid P \in \Spec(R)\}
\]
and 
\[
\ds r_F(R) = \min\{r_F(R_P) \mid \sr(R) = \sr(R_P)\}.
\]
\item $r_F(R)>0$ if and only if $R$ is F-pure.
\item If $R$ is a positively graded algebra over a local ring $(R_0,\m_0)$, then $\sr(R) = \sr(R_\m)$ and $r_F(R) = r_F(R_\m)$, where $\m = \m_0+R_{>0}$.
\end{enumerate}
\end{Theoremx}

We remark that Theorem~\ref{THMx betas} and Theorem~\ref{THMx r_F} are here only stated for global Frobenius Betti numbers and F-splitting ratio of the ring $R$. In Section \ref{Section betas} and \ref{Section global r_F}, we actually obtain more general results for global Frobenius Betti numbers and F-splitting ratio of finitely generated $R$-modules.

In the final section of this article, we positively answer \cite[Question 4.24]{DSPY}. In the local case, it was proved in \cite{BST2012} that the F-signature of a Cartier algebra $\D$ on $R$ is positive if and only if the pair $(R,\D)$ is strongly F-regular. These authors were able to recover the same result in the global setting, provided the Cartier algebra $\D$ satisfies certain additional assumptions \cite[Theorem 2.24]{DSPY}. We are able to remove these extra conditions:
\begin{Theoremx}
Let $R$ be an F-finite domain, and $\D$ be a Cartier algebra on $R$. Then $(R,\D)$ is strongly F-regular if and only if $\s(R,\D)>0$.
\end{Theoremx}

\section{Background on Frobenius Betti numbers of local rings}\label{Background}

%\section{Background on Frobenius Betti numbers of local rings}\label{Frobenius Betti Numbers} 
%The Frobenius Betti numbers of a finitely generated module over a local ring $(R,\m,k)$ can be viewed as a higher degree version of the Hilbert-Kunz multiplicity. 

%Since this will be used later in this section, we recall the aforementioned theorem of Aberbach and Li, which connects Frobenius Betti numbers with the singularities of the ring.

%\begin{Theorem}\cite[Theorem ??]{AberbachLi} Let $(R,\m,k)$ be an F-finite local ring. $R$ is regular if and only if $\beta_i^F(R)=0$ for some (equivalently, for all) $i \geq 1$.
%\end{Theorem}
 
%Frobenius Betti numbers provide a measure of the asymptotic growth of the Betti numbers of $F^e_*R$. Even though this terminology has been introduced in \cite{DSHNB}, there are several results in the literature regarding these invariants, preceding this article. In particular, results of Aberbach and Li \cite{AberbachLi,Li2008} relate Frobenius Betti numbers with the singularities of the ring. We now recall the definition of these invariants in the local setting.

Let $(R,\m,k)$ be an F-finite local ring of dimension $d$. For an $R$-module $M$, we denote by $F^e_*M$ the module structure on $M$ induced by restriction of scalars via $F^e:R \to R$, the $e$-th iterate of the Frobenius endomorphism on $R$. For any $R$-module $L$ of finite length, and any finitely generated $R$-module $M$, the $R$-module $\Tor_i^R(L,F^e_*M)$ has finite length, for all $i \geq 0$ and all $e \in \Z_{>0}$. For a prime ideal $P$ of $R$, let $\kappa(P)$ denote the residue field $R_P/PR_P \cong (R/P)_P$ of the local ring $R_P$. We set $\gamma(R)=\max\{\alpha(P)\mid P\in \Spec(R)\}$, where $\alpha(P) = \log_p[F_*\kappa(P):\kappa(P)]$. Moreover, given a finitely generated $R$-module $M$, we set $\gamma(M) = \gamma(R/\ann(M))$. In \cite{Seibert}, Seibert proves that the limit
\[
\ds \beta_i^F(L,M,\gamma) = \lim_{e \to \infty} \frac{\lambda_R\left(\Tor_i^R(L,F^e_*M)\right)}{p^{e\gamma}}
\]
exists for every integer $\gamma \geq \gamma(M)$. %Note that, if $I$ is an $\m$-primary ideal, then $\beta_0^F(R/I,R,\gamma(R))$ is just $\eHK(I)$, the Hilbert-Kunz multiplicity of the ideal $I$. 
When $L=R/\m$, we simply denote $\beta_i^F(R/\m,M,\gamma)$ by $\beta_i^F(M,\gamma)$. In addition, if $\gamma=\gamma(M)$, we only write $\beta_i^F(M)$, omitting $\gamma$ from the notation, and we call this invariant the {\it $i$-th Frobenius Betti number of $M$}. We warn the reader that, in \cite{DSHNB}, the normalization factor in the denominator of the $i$-th Frobenius Betti number is chosen to be $p^{e\gamma(R)}$, rather than $p^{e\gamma(M)}$. 

Observe that, for all $e$, the length $\lambda_R(\Tor_i^R(R/\m,F^e_*M))$ in the numerator of $\beta_i^F(M)$ is the $i$-th Betti number of the $R$-module $F^e_*M$. Moreover
\[
\ds \beta_0^F(M,\gamma(R)) = \lim\limits_{e \to \infty} \frac{\lambda_R(R/\m \otimes_R F^e_*M)}{p^{e\gamma(R)}} = \eHK(M)
\]
is the Hilbert-Kunz multiplicity of $M$, with respect to the maximal ideal $\m$. %When $R$ is not F-finite, one can view $F^e_*R$ as an $R$-algebra, and define
%\[
%\ds \beta_i^F(M) = \lim_{e \to \infty} \frac{\lambda_{F^e_*R}(\Tor_i^R(k,F^e_* M))}{p^{e \dim(R)}}.
%\] 

To ease up the notation, we often write $\beta_i(e,\m,M)$ for $\lambda_R(\Tor_i^R(R/\m,F^e_*M))$. More generally, for $P \in \Spec(R)$ and integers $e, i \geq 0$, we define 
\[
\ds \beta_i(e,P,M)= \lambda_{R_P}(\Tor_i^{R_P}(\kappa(P),F^e_*(M_P))).
\] 
With this notation, the $i$-th Frobenius Betti number of $M_P$ as an $R_P$-module is 
\[
\ds \beta_i^F(M_P) = \lim_{e \to \infty} \frac{\beta_i(e,P,M)}{p^{e\gamma(M_P)}}. 
\]
\begin{Remark} We warn the reader about a potential source of confusion with our notation. 
If we view $M_P$ as an $R$-module, $\beta_i^F(M_P)$ is equal to $\lim\limits_{e \to \infty} \beta_i(e,\m,M_P)/p^{e\gamma(M_P)}$. On the other hand, viewing $M_P$ as an $R_P$-module, $\beta_i^F(M_P)$ is equal to $\lim\limits_{e \to \infty} \beta_i(e,P,M_P)/p^{e\gamma(M_P)}$. We could fix the problem by specifying the underlying ring; however, when writing $\beta_i^F(M_P)$ for a finitely generated $R$-module $M$, we will always view $M_P$ as an $R_P$-module. Therefore, we do not specify the underlying ring here, to avoid making the notation heavier. 
\end{Remark}

\section{Uniform convergence and upper semi-continuity results}  \label{Section Uniform convergence Frob Betti}

A key ingredient that is used in \cite{DSPY} for developing a global theory of Hilbert-Kunz multiplicity and F-signature are certain semicontinuity results. In particular, to relate the global Hilbert-Kunz multiplicity to the invariants in the localization, the upper semi-continuity of the functions 
\[
\xymatrixcolsep{2mm}
\xymatrixrowsep{2mm}
\xymatrix{
\lambda_e:\Spec(R)\ar[r] &  \R \\ 
P\ar[r] & \lambda(M_P/P^{[p^e]}M_P)/p^{e\height(P)}
}
\]
for locally equidimensional excellent rings, and the uniform convergence to their limit, play a crucual role. The upper semi-continuity was first established in \cite{Kunz1976} (Kunz claimed that this result was true for an equidimensional ring, but Shepherd-Barron noted in \cite{ShepherdBarron} that the locally equidimensional assumption is needed). The uniform convergence was established in \cite[Theorem 5.1]{Polstra2018}. An immediate consequence of these two facts is that the Hilbert-Kunz function is upper semi-continuous on the spectrum of a locally equidimensional ring \cite{Smirnov2015,Polstra2018}.

As for the Hilbert-Kunz multiplicity, $\beta_i^F(-)$ is additive on short exact sequences, therefore for several arguments we can reduce to the case $M=R/\q$, where $\q$ is a prime. We warn the reader that this does not allow us to reduce to the case when $R$ is a domain, as for the Hilbert-Kunz multiplicity. In fact, we still have to compute the lengths modules $\Tor_i^{R_P}(\kappa(P),F^e_*(R/\q)_P)$ over the rings $R_P$. This complicates the arguments for the Frobenius Betti numbers, and adds some technical work to the proofs of uniform convergence and upper semi-continuity in this section.

%Let $R$ be a locally equidimensional F-finite ring such that $\alpha(P) + \height(P) = \gamma$ is constant on $\Spec(R)$. 
%We now recall the local definition of Frobenius Betti numbers. Let $(R,\m,k)$ be an F-finite local ring. For a finitely generated $R$-module $M$, let $\beta_i^R(M)$ denote the $i$-th Betti number of $M$, that is, $\beta_i^R(M) = \lambda_R(\Tor_i^R(k,M))$. The {\it $i$-th Frobenius Betti number} of a finitely generated $R$-module $M$ is defined to be 
%\[
%\ds \beta_i^F(M) = \lim_{e \to \infty} \frac{\lambda_R(\Tor_i^R(k,F^e_* M))}{p^{e \gamma(R)}} = \lim_{e \to \infty} \frac{\beta_i^R(F^e_*M)}{p^{e\gamma(R)}}.
%\] 

%Note that $\beta_0^F(M) = e_{HK}(M)$. 

Let $X$ be a topological space. Recall that a function $f:X \to \R$ is {\it upper semi-continuous} if for all $x \in X$, for all $\varepsilon > 0$, there exists an open set $U \subseteq X$ containing $x$ such that $f(y)<f(x)+\varepsilon$ for all $y \in U$. We say that $f$ is {\it dense upper semi-continuous} if, for all $x \in X$, there exists an open set $U \subseteq X$ containing $x$ such that $f(y) \leq f(x)$ for all $y \in U$.

In what follows, it will be helpful to consider an Euler characteristic version of the Frobenius Betti numbers, in part inspired by Dutta multiplicities \cite{Dutta1983}. For integers $i, e \geq 0$, a finitely generated $R$-module $M$, and a prime $P \in \Spec(R)$, recall that we have defined $\beta_i(e,P,M) = \lambda_{R_P}(\Tor_i^{R_P}(\kappa(P),F^e_*(M_P)))$. In the same setup, we let
\[
\ds \chi_i(e,P,M) = \sum_{j=0}^i (-1)^{i-j} \beta_j(e,P,M).
\] 

\begin{Proposition} \label{dense upper semicont} 
Let $R$ be an F-finite ring, and $M$ be a finitely generated $R$-module. For $i \in \Z_{\geq 0}$, and $e \in \Z_{>0}$, the functions
\[
\ds P \in \Spec(R) \mapsto \beta_i(e,P,M) \hspace{0.5cm} \mbox{ and } \hspace{0.5cm} P \in \Spec(R) \mapsto \chi_i(e,P,M)
\]
are dense upper semi-continuous. In particular, they are upper semi-continuous.
\end{Proposition}
\begin{proof}
%It suffices to show that the functions $P \mapsto \beta_i(e,P,M)$ and $P \mapsto \chi_i(e,P,M)$ are dense upper semi-continuous, since this property is not affected by dividing by $p^{e\gamma(R)}$.

Let $P \in \Spec(R)$, and let $e > 0$. Consider a minimal free resolution of $F^e_*(M_P)$:
\[
\xymatrix{
\ldots \ar[r] &  R_P^{\beta_i(e)} \ar[r]^-{\varphi_i} & R_P^{\beta_{i-1}(e)} \ar[r]^-{\varphi_{i-1}} & \ldots\ldots \ar[r] & R_P^{\beta_0(e)} \ar[r]^-{\varphi_0} & F^e_*(M_P) \ar[r] & 0,
}
\]
where $\beta_j(e) = \beta_j(e,P,M)$ is the $j$-th Betti number of $F^e_*(M_P)$ as an $R_P$-module. Since the rank of each free $R_P$-module in the resolution is finite, for all $j$ we can find lifts $\psi_j \in \Hom_R(R^{\beta_j(e)},R^{\beta_{j-1}(e)})$ of $\varphi_j$ from $R_P$ to $R$, giving maps
\[
\xymatrix{
\ldots \ar[r] &  R^{\beta_i(e)} \ar[r]^-{\psi_i} & R^{\beta_{i-1}(e)} \ar[r]^-{\psi_{i-1}} & \ldots\ldots \ar[r] & R^{\beta_0(e)} \ar[r]^-{\psi_0} & F^e_*M \ar[r] & 0.
}
\]
Note that this is not even necessarily a complex. However, since $R$ is Noetherian and all the free modules in question have finite rank, by inverting an element $s \in R \smallsetminus P$ we can assume that $\ker((\psi_j)_s) = \im((\psi_{j+1})_s)$ for all $j=0,\ldots,i$, and that $\im((\psi_0)_s) = F^e_*(M_s)$. In other words,
\[
\xymatrix{
R_s^{\beta_{i+1}(e)} \ar[r]^-{(\psi_{i+1})_s} &  R_s^{\beta_i(e)} \ar[r]^-{(\psi_i)_s} & R_s^{\beta_{i-1}(e)} \ar[r]^-{(\psi_{i-1})_s} & \ldots\ldots \ar[r] & R_s^{\beta_0(e)} \ar[r]^-{(\psi_0)_s} & F^e_*(M_s) \ar[r] & 0.
}
\]
is part of a free resolution of $F^e_*(M_s)$ over the ring $R_s$. In particular, by localizing at any prime $Q \in \Spec(R)$ not containing $s$, the complex is still exact, and it becomes a free resolution of $F^e_*(M_Q)$ over $R_Q$. However, it may not be minimal. That is, $\beta_i(e,Q,M) \leq \beta_i$ for all $i \geq 0$. If we consider the Zariski open set $D(s)= \{Q \in \Spec(R) \mid s \notin Q\}$, we therefore have that $\beta_i(e,Q,M) \leq \beta_i(e,P,M)$ for all $Q \in D(s)$. This shows dense upper semi-continuity of the function $P \mapsto \beta_i(e,P,M)$. We now focus on the function $P \mapsto \chi_i(e,P,M)$. Let $P,\varphi_j, \psi_j$ and $s \in R \smallsetminus P$ be as above. Let $Q \in D(s)$, and denote $\Omega_j=\ker((\psi_{j-1})_Q)$, for all $j=0,\ldots,i$. This gives short exact sequences of $R_Q$-modules:
\[
\ds 0 \longrightarrow \Omega_j \longrightarrow R_Q^{\beta_{j-1}(e)} \longrightarrow \Omega_{j-1} \longrightarrow 0
\]
for all $j=1,\ldots,i$. Tensoring with $\kappa(Q)$, we obtain long exact sequences
\[
\ds 0 \longrightarrow \Tor_1^{R_Q}(\kappa(Q),\Omega_{j-1}) \longrightarrow \Omega_j/Q\Omega_j \longrightarrow \kappa(Q)^{\beta_{j-1}(e)} \longrightarrow \Omega_{j-1}/Q\Omega_{j-1} \longrightarrow 0.
\]
Finally, since $\Tor_1^{R_Q}(\kappa(Q),\Omega_{j-1}) \cong \Tor_j^{R_Q}(\kappa(Q),F^e_*(M_Q))$, by counting lengths over $R_Q$ we obtain that
\[
\ds \mu_{R_Q}(\Omega_i) + \sum_{j=0}^{i-1} (-1)^{i-j} \beta_j(e) = \sum_{j=0}^i (-1)^{i-j} \lambda_{R_Q}(\Tor_j^{R_Q}(\kappa(Q),F^e_*(M_Q))) = \chi_i(e,Q,M),
\]
where $\mu_{R_Q}(-)$ denotes the minimal number of generators of an $R_Q$-module. Since $\mu_{R_Q}(\Omega_i) \leq \beta_i(e)$, we get the desired conclusion
\[
\ds \chi_i(e,Q,M) \leq \beta_i(e) + \sum_{j=0}^{i-1} (-1)^{i-j} \beta_j(e) = \chi_i(e,P,M). \qedhere
\]
\end{proof}
We recall the following global version of an observation made by Dutta in \cite{Dutta1983}.

\begin{Lemma}[{\cite[Lemma 2.2]{Polstra2018}}]\label{Dutta's Lemma} Let $R$ be an F-finite domain. There exists a finite set of nonzero primes $\mathcal{S}(R)$, and a constant $C$, such that for every $e\in \Z_{>0}$
\begin{enumerate}
\item there is a containment of $R$-modules $R^{\oplus p^{e\gamma(R)}}\subseteq F^e_*R$,
\item which has a prime filtration with prime factors isomorphic to $R/Q$, where $Q\in\mathcal{S}(R)$,
\item and for each $Q\in \mathcal{S}(R)$, the prime factor $R/Q$ appears no more than $Cp^{e\gamma(R)}$ times in the chosen prime filtration of $R^{\oplus p^{e\gamma(R)}}\subseteq F^e_*R$.
\end{enumerate}
\end{Lemma}

Lemma~\ref{Dutta's Lemma} is used by the second author in \cite{Polstra2018} to establish the presence of strong uniform bounds for all F-finite rings, and in \cite{DSPY} to show the existence of global Hilbert-Kunz multiplicity and F-signature. 
\begin{Lemma} \label{uniform convergence primes} Let $R$ be an F-finite ring, $\q \in \Spec(R)$, $i$ and $\gamma$ be non-negative integers, with $\gamma \geq \gamma(R/\q)$. There exists a constant $A$, that only depends on $i$ and $\q$, such that
\[
\ds \left| \frac{\beta_i(e_1+e_2,P,R/\q)}{(q_1q_2)^{\gamma}} - \frac{\beta_i(e_2,P,R/\q)}{q_2^{\gamma}}\right| \leq \frac{A}{q_2}
\]
for all $q_1=p^{e_1}$, $q_2=p^{e_2}$ and $P \in \Spec(R)$. In particular, the sequence $\ds \left\{\frac{\beta_i(e,\bullet, R/\q)}{p^{e\gamma}}\right\}_{e \in \Z_{>0}}$ converges uniformly on $\Spec(R)$.
\end{Lemma}
\begin{proof}
Let $\gamma':= \gamma(R/\q)$. Note that, for all $P \in \Spec(R)$, the limit
\[
\ds \lim_{e \to \infty} \frac{\beta_i(e,P,R/\q)}{p^{e\gamma'}}
\]
exists and it is finite by \cite{Seibert}. We will first show that $\frac{\beta_i(e,\bullet,R/\q)}{p^{e\gamma'}}$ converges uniformly to this limit. Let $q_1=p^{e_1}$. Consider a set of primes $\mathcal{S}(R/\q)$ as in Lemma \ref{Dutta's Lemma} for the inclusion $(R/\q)^{q_1^{\gamma'}} \subseteq F^{e_1}_*(R/\q)$, and let $C$ be the constant given by the Lemma. For each $\p \in \mathcal{S}(R/\q)$, the ring $R/\p$ is an F-finite domain with $\gamma(R/\p) \leq \gamma'-1$. Applying Lemma \ref{Dutta's Lemma} to each $R/\p$, we obtain lists $\mathcal{S}(R/\p)$ and constants $D_\p$. Let $\mathcal{S} = \bigcup_{\p \in \mathcal{S}(R/\q)} \mathcal{S}(R/\p)$, and let $D = \max\{D_\p \mid \p \in \mathcal{S}\}$. Note that, for all $\p \in \mathcal{S}(R/\q)$ and all $q_2=p^{e_2}$, the inclusion $(R/\p)^{q_2^{\gamma(R/\p)}} \subseteq F^{e_2}_*(R/\p)$ has a filtration by cyclic modules of the form $R/\a$, with $\a$ inside $\mathcal{S}$. Furthermore, each ideal in such a filtration appears at most $Dq_2^{\gamma(R/\p)} \leq D q_2^{\gamma'-1}$ times.
For an integer $j \geq 0$ and a prime $\p' \in \Spec(R)$, let $E_{j,\p'}$ be the minimal number of generators of a $j$-th syzygy of $R/\p'$ over $R$. Let $E_j:=\max\{E_{j,\p'} \mid \p' \in \mathcal{S}\}$. Note that $\lambda_{R_P}(\Tor_j^{R_P}(\kappa(P),(R/\p')_P))\leq E_j$ for all $\p' \in \mathcal{S}$ and all primes $P \in \Spec(R)$, since after localizing a resolution of $R/\p'$ over $R$ at $P$ it stays exact, but may not be minimal. 

Now let $P \in \Spec(R)$. After localizing everything at $P$ we still have filtrations as before, possibly with fewer factors, since some of them may have collapsed. We remove from the lists $\mathcal{S}(R/\q)$ and $\mathcal{S}$ prime ideals $\p$ such that $\p R_P = R_P$; we still call the new lists $\mathcal{S}(R/\q)$ and $\mathcal{S}$. If $(R/\q)_P = 0$ there is nothing to show, so let us assume that $(R/\q)_P \ne 0$. Consider the short exact sequence
\[
\ds 0 \longrightarrow  (R/\q)_P^{q_1^{\gamma'}} \longrightarrow F^{e_1}_*(R/\q)_P \longrightarrow T(q_1)_P \longrightarrow 0
\]
where $T(q_1)_P$ are $(R/\q)_P$-modules of dimension at most $\dim(R/\q)-1$. The functor $F^{e_2}_*$ is exact, and yields a short exact sequence
\[
\ds 0 \longrightarrow (F^{e_2}_*(R/\q)_P)^{q_1^{\gamma'}} \longrightarrow F^{e_1+e_2}_*(R/\q)_P \longrightarrow F^{e_2}_*T(q_1)_P \longrightarrow 0
\]
Applying $\Tor^{R_P}(\kappa(P),-)$ and counting lengths we obtain
\[
\ds \left|\lambda\left(\Tor_i^{R_P}(\kappa(P),F^{e_1+e_2}_*(R/\q)_P)\right) - p^{e_1\gamma'} \lambda\left(\Tor_i^{R_P}(\kappa(P),F^{e_2}_*(R/\q)_P)\right) \right| \leq\]
\[
\ds \leq  \max\left\{ \lambda\left(\Tor_i^{R_P}(\kappa(P),F^{e_2}_*T(q_1)_P)\right),  \lambda\left(\Tor_{i+1}^{R_P}(\kappa(P),F^{e_2}_*T(q_1)_P)\right) \right\}.
\]
Equivalently, we obtain that
\[
\ds \left| \beta_i(e_1+e_2,P,R/\q) - p^{e_1\gamma'} \beta_i(e_2,P,R/\q) \right| \leq \max\left\{ \beta_i(e_2,P,T(q_1)),\beta_{i+1}(e_2,P,T(q_1))\right\}.
\]
The modules $T(q_1)_P$ have filtrations $0 \subseteq T_1 \subseteq \ldots \subseteq T_{i(q_1)} = T(q_1)_P$ as in Lemma \ref{Dutta's Lemma}, and by exactness of $F^{e_2}_*$ we then have filtrations $0 \subseteq F^{e_2}_*T_1 \subseteq \ldots \subseteq F^{e_2}_*T_{i(q_1)} = F^{e_2}_*T(q_1)_P$. The relative quotients are isomorphic to $F^{e_2}_*(R/\p)_P$, for some $\p \in \mathcal{S}$ appearing at most $Cq_1^{\gamma'}$ times in the filtration. Applying $\Tor^{R_P}_\bullet(\kappa(P),-)$ to the resulting short exact sequences, for all $j$ we then have that 
\[
\ds \beta_j(e_2,P,T(q_1)) \leq C\left|\mathcal{S}(R/\q)\right| q_1^{\gamma'} \max\left\{\beta_j(e_2,P,R/\p) \mid \p \in \mathcal{S}(R/\q)\right\}.
\]
For $\p \in \mathcal{S}(R/\q)$, the inclusion $(R/\p)_P^{q_2^{\gamma(R/\p)}} \subseteq F^{e_2}_*(R/\p)_P$ has a filtration by prime ideals in $\mathcal{S}$ appearing at most $Dq_2^{\gamma'-1}$ times. Applying $\Tor^{R_P}_\bullet(\kappa(P),-)$ to the resulting short exact sequences, we obtain that for all $\p \in \mathcal{S}(R/\q)$ and all $j \geq 0$
\[
\ds  \beta_j(e_2,P,R/\p) \leq D\left|\mathcal{S}\right|q_2^{\gamma'-1}\max\{\beta_j(0,P,R/\p') \mid \p' \in \mathcal{S} \} \leq D\left|\mathcal{S}\right|E_jq_2^{\gamma'-1},
\]
where $\beta_j(0,P,R/\p')$ is just the $j$-th Betti number of the $R_P$-module $F^0_*((R/\p')_P) = (R/\p')_P$, that is, $\lambda_{R_P}(\Tor_j^{R_P}(\kappa(P),(R/\p')_P))$. Recall that the constants $C,D,E_j$ are completely independent of $q$, $q'$, and the prime $P$. Set $A:= CD\left|\mathcal{S}\right| \left|\mathcal{S}'\right|\max\{E_i,E_{i+1}\}$ and divide by $(q_1q_2)^{\gamma'}$, to obtain
\[
\ds \left| \frac{\beta_i(e_1+e_2,P,R/\q)}{(q_1q_2)^{\gamma'}} - \frac{\beta_i(e_2,P,R/\q)}{q_2^{\gamma'}}\right| \leq \frac{A}{q_2}
\]
for all $q_1,q_2$, for all $P \in \Spec(R)$. This shows that $\frac{\beta_i(e,\bullet,R/\q)}{p^{e\gamma'}}$ converges uniformly.

If $\gamma = \gamma'$ then there is nothing else to prove. If $\gamma>\gamma'$, then the sequence 
\[
\ds \frac{\beta_i(e,P,R/\q)}{p^{e\gamma}} = \frac{\beta_i(e,P,R/\q)}{p^{e\gamma'}} \cdot \frac{1}{p^{e(\gamma-\gamma')}}
\]
converges to zero. Furthermore, the convergence is still uniform, and to see this it is enough to show that the limit function $\lim\limits_{e \to \infty} \frac{\beta_i(e,\bullet,R/\q)}{p^{e\gamma'}}$ is bounded on $\Spec(R)$. To see that, observe that, combining uniform convergence of the sequence $\frac{\beta_i(e,\bullet,R/\q)}{p^{e\gamma'}}$ with Proposition \ref{dense upper semicont}, we obtain that $P \mapsto \lim\limits_{e \to \infty} \frac{\beta_i(e,P,R/\q)}{p^{e\gamma'}}$ is upper semi-continuous. Finally, by quasi-compactness of $\Spec(R)$, we conclude that
\[
\ds \sup\left\{\lim_{e \to \infty} \frac{\beta_i(e,P,R/\q)}{p^{e\gamma'}} \mid P \in \Spec(R)\right\} = \max\left\{\lim_{e \to \infty} \frac{\beta_i(e,P,R/\q)}{p^{e\gamma'}} \mid P \in \Spec(R)\right\}< \infty. \qedhere
\]
\end{proof}

\begin{Theorem} \label{uniform convergence general} Let $R$ be an F-finite ring, and let $M$ be a finitely generated $R$-module. Let $\gamma$ be an integer satisfying $\gamma \geq \gamma(M)$. For any fixed $i \in \Z_{\geq 0}$, the sequence of functions
\[
\xymatrixrowsep{1mm}
\xymatrix{
\Spec(R) \ar[r] & \R \\
P \ar[r] & \ds \frac{\beta_i(e,P,M)}{p^{e\gamma}}  
}
\]
is uniformly bounded over $\Spec(R)$, and converges uniformly to its limits as $e \to \infty$.
\end{Theorem}
\begin{proof}
We proceed by induction on $\gamma(M)$. If $\gamma(M) = 0$, then $\Supp(M)$ consists of finitely many maximal ideals $\m_1,\ldots,\m_r$. In addition, there exists $e_0$, depending on $M$, such that $\ann(F^e_*M) = \m_1\cap \ldots \cap \m_t$ for all $e \geq e_0$. By the Chinese Remainder Theorem, we have that $F^e_*M \cong \oplus_{j=1}^r F^e_*(R/\m_j)^{\lambda_{R_{\m_j}}(M_{\m_j})}$ for all $e \geq e_0$. Thus, for $e \geq e_0$, we have
\[
\ds \frac{\beta_i(e,\bullet,M)}{p^{e\gamma}} = \sum_{j=1}^r \lambda_{R_{\m_j}}(M_{\m_j})\frac{\beta_i(e,\bullet,R/\m_j)}{p^{e\gamma}}.
\]
Uniform convergence of $\frac{\beta_i(e,\bullet,M)}{p^{e\gamma}}$ then follows from Lemma \ref{uniform convergence primes}, since the first $e_0-1$ terms do not affect this kind of considerations. Furthermore, by Proposition~\ref{dense upper semicont}, the function $\frac{\beta_i(e,\bullet,M)}{p^{e\gamma}}$ is bounded on $\Spec(R)$ for every fixed $e \in \Z_{>0}$, as a consequence of its upper semi-continuity and of quasi-compactness of $\Spec(R)$. It then follows that $\left\{\frac{\beta_i(e,\bullet,R/\m_j)}{p^{e\gamma}}\right\}_{e \in \Z_{>0}}$ is uniformly bounded on $\Spec(R)$ for each $j$, and thus so is $\left\{\frac{\beta_i(e,\bullet,M)}{p^{e\gamma}}\right\}_{e\in \Z_{>0}}$.

Now assume that $\gamma(M) > 0$, and suppose that $\ann(M)$ is radical first. Consider a prime filtration of $M$:
\[
\ds 0 = M_0 \subseteq M_1 \subseteq M_2 \subseteq \ldots \subseteq M_t = M,
\]
where $M_j/M_{j-1} \cong R/P_j$ for some $P_j \in \Spec(R)$, for $j=1,\ldots,t$. Consider the $R$-module $N:= \bigoplus_{j=1}^t R/P_j$, and let $W = R \smallsetminus \bigcup\{\p \mid \p \in \Min(M)\}$. Since $\ann(M)$ is radical, we have an isomorphism $M_W \cong \prod_{\p \in \Min(M)} \kappa(\p)^{\lambda_{R_\p}(M_\p)} \cong N_W$ of $R_W$-modules. Because $M$ and $N$ are finitely generated over $R$, we can find an $R$-linear map $\varphi:M \to N$ that, after localizing at $W$, becomes an isomorphism. In addition, if we write
\[
\ds 0 \longrightarrow K \longrightarrow M \stackrel{\varphi}{\longrightarrow} N \longrightarrow C \longrightarrow 0,
\]
we have that $\gamma(K)$ and $\gamma(C)$ are at most $\gamma(M)-1$. Denote by $T$ the image of $\varphi$. After applying the functors $F^e_*(-)$ and $\Tor^{R_\bullet}(\kappa(\bullet),-_\bullet)$, and comparing lengths, we obtain that
\[
\ds \left| \beta_i(e,\bullet,M) - \beta_i(e,\bullet,T) \right| \leq \max\{ \beta_i(e,\bullet,K), \beta_{i-1}(e,\bullet,K)\}
\]
and that 
\[
\ds \left| \beta_i(e,\bullet,T) - \beta_i(e,\bullet,N) \right| \leq \max\{ \beta_i(e,\bullet,C), \beta_{i+1}(e,\bullet,C)\}.
\]
By the triangle inequality, we get that
\[
\ds \left| \beta_i(e,\bullet,M) - \beta_i(e,\bullet,N) \right| \leq n_i(\bullet),
\]
where $n_i(\bullet) = \max\{\beta_i(e,\bullet,K), \beta_{i-1}(e,\bullet,K)\} + \max\{ \beta_i(e,\bullet,C), \beta_{i+1}(e,\bullet,C)\}$. By inductive hypothesis, we have that the sequences 
\[
\ds \frac{\beta_i(e,\bullet,K)}{p^{e\gamma}}, \ \ \frac{\beta_{i-1}i(e,\bullet,K)}{p^{e\gamma}}, \ \ \frac{\beta_i(e,\bullet,C)}{p^{e\gamma}}, \ \ \frac{\beta_{i+1}(e,\bullet,C)}{p^{e\gamma}}
\]
are uniformly bounded, and converge uniformly on $\Spec(R)$ as $e \to \infty$. Therefore, the sequence of functions $\frac{n_i(\bullet)}{p^{e\gamma}}$ given by the sum of the maxima as defined above satisfies the same properties. In addition, since $\gamma(K)$ and $\gamma(C)$ are at most $\gamma(M)-1 < \gamma$, we have that $\frac{n_i(\bullet)}{p^{e\gamma}}$ converges to zero uniformly. In particular, we have that
\[
\ds  \beta_i^F(M_\bullet,\gamma) = \lim_{e \to \infty} \frac{\beta_i(e,\bullet,M)}{p^{e\gamma}} = \lim_{e \to \infty} \frac{\beta_i(e,\bullet,N)}{p^{e\gamma}}.
\]
Since $F^e_*N \cong \oplus_{j=1}^t F^e_*(R/P_j)$, the sequence $\left\{\frac{\beta_i(e,\bullet,N)}{p^{e\gamma}}\right\}$ converges uniformly by Lemma \ref{uniform convergence primes}. Therefore, for all $\varepsilon > 0$, there exists $e_1>0$ such that for all $e > e_1$ we have
\begin{itemize}
\item $\ds \left| \frac{\beta_i(e,\bullet,N)}{p^{e\gamma}} - \beta_i^F(M_\bullet) \right| < \frac{\varepsilon}{2}$.
\item $\ds \frac{n_i(\bullet)}{p^{e\gamma}} < \frac{\varepsilon}{2}$.
\end{itemize}
By the triangle inequality, we obtain
\[
\ds \left| \frac{\beta_i(e,\bullet,M)}{p^{e\gamma}} - \beta_i^F(M_\bullet) \right| \leq \left| \frac{\beta_i(e,\bullet,N)}{p^{e\gamma}} - \beta_i^F(M_\bullet) \right| + \frac{n_i(\bullet)}{p^{e\gamma}} < \varepsilon,
\]
that is, $\frac{\beta_i(e,\bullet,M)}{p^{e\gamma}}$ converges uniformly. Finally, the fact that $\left\{\frac{\beta_i(e,\bullet,M)}{p^{e\gamma}}\right\}_{e \in \Z_{>0}}$ is uniformly bounded on $\Spec(R)$ follows from Proposition \ref{dense upper semicont}, as for the case $\gamma(M)=0$.

If $R/\ann(M)$ is not reduced, we can find $e_0>0$ such that $\ann(F^{e}_*M)=\sqrt{\ann(M)}$ for all $e\geq e_0$. Consider $M':=F^{e_0}_*M$, and note that $F^e_*M'  \cong F^{e+e_0}_*M$ for all $e \geq 0$. In addition, $R/\ann(M')$ is reduced. We replace the sequence $\ds\frac{\beta_i(e,\bullet,M)}{p^{e\gamma}}$ with the sequence $\ds \frac{\beta_i(e,\bullet,M')}{p^{e\gamma}}$. Since they only differ by finitely many terms, and by a correction term of $p^{e_0\gamma}$, uniform convergence and uniform boundedness of the former would follow those for the latter. Given that $\ann(M')$ is radical, this has been proved above.
\end{proof}
Let $i \geq 0$ be an integer, $M$ be a finitely generated $R$-module, and $\gamma \geq \gamma(M)$ be an integer. Using the notation introduced in Section \ref{Background}, we let $\beta_i^F(M_\bullet,\gamma)$ be the limit function of the sequence considered in Theorem~\ref{uniform convergence general}. Recall that, for $P \in \Spec(R)$, the $i$-th Frobenius Betti number of the $R_P$-module $M_P$ is
\[
\ds \beta_i^F(M_P) = \lim_{e \to \infty} \frac{\beta_i(e,P,M)}{p^{e\gamma(M_P)}}.
\]
The difference between $\beta_i^F(M_P,\gamma)$ and $\beta_i^F(M_P)$ is a possibly different normalization in the denominator. More specifically, let $Z_{M,\gamma} = \{P \in \Spec(R) \mid \gamma(M_P) = \gamma\}$. The set $Z_{M,\gamma}$ in the case $M=R$ and $\gamma=\gamma(R)$ has been introduced in \cite{DSPY} to study relations between local and global invariants for general F-finite rings. Clearly one has
%\[
%\widetilde{\beta_i^F}(M_P) = \left\{\begin{array}{ll} \beta_i^F(M_P) & \mbox{ if } P \in Z_R \\ \\ 0 & \mbox{ otherwise} \end{array} \right.
%\]
%In particular, 
$\beta_i^F(M_P,\gamma) = \beta_i^F(M_P)$ whenever $P \in Z_{M,\gamma}$. On the other hand, one has $\beta_i^F(M_P,\gamma)=0$ if $P \notin Z_{M,\gamma}$. % (equality holds, for example, when $R$ is an F-finite domain). 
Similar considerations hold for the sequence of functions
\[
\xymatrixrowsep{1mm}
\xymatrix{
\Spec(R) \ar[r] & \R \\
P \ar[r] & \ds \frac{\chi_i(e,P,M)}{p^{e\gamma(R)}} 
}
\]
and its limit as $e \to \infty$, that we denote by $\chi_i^F(M_\bullet,\gamma):\Spec(R) \to \R$.

\begin{Corollary} Let $R$ be an F-finite ring, $M$ be a finitely generated $R$-module, and $\gamma$ be an integer satisfying $\gamma \geq \gamma(M)$. For any fixed $i \in \Z_{\geq 0}$, the functions 
\[
\xymatrixrowsep{1mm}
\xymatrixcolsep{2mm}
\xymatrix{ 
\Spec(R) \ar[r] & \R  & \ \mbox{ and } \ & \Spec(R) \ar[r] & \R  \\
P \ar[r] & \beta_i^F(M_P,\gamma) & & P \ar[r] & \chi_i^F(M_P,\gamma)
}
\]
are upper semi-continuous.
\end{Corollary}
\begin{proof}
Since dividing by $p^{e\gamma}$ does not affect semi-continuity of the functions $P \in \Spec(R) \mapsto \beta_i(e,P,M)$ and $P \in \Spec(R) \mapsto \chi_i(e,P,M)$, Proposition \ref{dense upper semicont} gives that $\frac{\beta_i(e,\bullet,M)}{p^{e\gamma}}$ and $\frac{\chi_i(e,\bullet,M)}{p^{e\gamma}}$ are upper semi-continuous for all $e \geq 0$. Because the second sequence is built as a finite alternating sum of elements from the first, Theorem \ref{uniform convergence general} gives uniform convergence over $\Spec(R)$ as $e \to \infty$ for both sequences. It then follows that $P \in \Spec(R) \mapsto \beta_i^F(M_P,\gamma)$ and $P \in \Spec(R) \mapsto \chi_i^F(M_P,\gamma)$ are upper semi-continuous, as they are the uniform limit of upper semi-continuous functions.
\end{proof}

\section{Global Frobenius Betti numbers and Frobenius Euler characteristic} \label{Section betas}

% The idea is to define a global $\chi_i$ starting from any ''minimal global'' resolution of $F^e_*M$. For an integer $e >0$, let $b_0(e)$ be the minimal number of generators of $F^e_*M$. We then have a surjection $R^{b_0(e)} \to F^e_*M \to 0$. The next Lemma shows that the minimal number of generators of the kernel of such a map is almost independent of the presentation.
In this section, we introduce and justify the notion of global Frobenius Betti numbers and Frobenius Euler characteristic. In what follows, $\mu_R(-)$ will denote the minimal number of generators of an $R$-module.

We start with an easy consequence of Schanuel's lemma.

\begin{Lemma} \label{Schanuel}
Let $R$ be a Noetherian ring of Krull dimension $d$, and let $M$ be a
finitely generated $R$-module.  Let 
\begin{align*}
0 \lra \Omega \lra R^{b_{i-1}} \lra \dotsb \dotsb \lra R^{b_0} \lra M \lra 0&\\
0 \lra \Omega' \lra R^{b_{i-1}'} \lra \dotsb \dotsb \lra R^{b_0'} \lra M \lra 0
\end{align*}
be exact sequences. Then 
$\left| \left(\mu_R(\Omega)+\sum_{j = 1}^i(-1)^jb_{i-j}\right)
-\left(\mu_R(\Omega') +\sum_{j = 1}^i(-1)^jb_{i-j}'\right)\right| \leq d$.
\end{Lemma}
\begin{proof}
By Schanuel's lemma, we have that 
\[
\Omega \oplus R^{\sum_{j \text{ odd}}b_{i-j}' + \sum_{j \text{ even}}b_{i-j}} \cong
\Omega' \oplus R^{\sum_{j \text{ odd}}b_{i-j} + \sum_{j \text{ even}}b_{i-j}'}.
\]
By the Forster-Swan Theorem \cite{Forster,Swan}, we may choose $\m \in \Max\Spec(R)$ such that $\mu_R(\Omega') \leq
\mu_{R_\m}(\Omega'_\m) + d$. Consequently we see
\begin{align*}
\mu_R(\Omega) + \sum_{j \text{ odd}}b_{i-j}' + \sum_{j \text{ even}}b_{i-j}
& \geq \mu_R(\Omega \oplus R^{\sum_{j \text{ odd}}b_{i-j}' + \sum_{j \text{ even}}b_{i-j}}) \\
&=  \mu_R(\Omega' \oplus R^{\sum_{j \text{ odd}}b_{i-j} + \sum_{j \text{ even}}b_{i-j}'}) \\
&\geq \mu_{R_\m}((\Omega'\oplus R^{\sum_{j \text{ odd}}b_{i-j} + \sum_{j \text{ even}}b_{i-j}'})_\m) \\
&= \mu_{R_\m}(\Omega'_\m) + {\sum_{j \text{ odd}}b_{i-j} + \sum_{j \text{ even}}b_{i-j}'}  \\
&\geq \mu_R(\Omega')-d+ {\sum_{j \text{ odd}}b_{i-j} + \sum_{j \text{ even}}b_{i-j}'}.
\end{align*}
Therefore $\left(\mu_R(\Omega')+\sum_{j = 1}^i(-1)^jb_{i-j}'\right)
-\left(\mu_R(\Omega) +\sum_{j = 1}^i(-1)^jb_{i-j}\right) \leq d$. 
Using a symmetric argument we establish the Lemma. 
\end{proof}

\begin{Remark} \label{Remark independence chi}
In the notation of Lemma \ref{Schanuel}, let $\gamma \geq\min\{1,d\}$ be an integer. For every $e \in \Z_{>0}$, fix a free resolution 
\[
\cdots \longrightarrow  R^{b_i(e)} \stackrel{\varphi_i(e)}{\longrightarrow} R^{b_{i-1}(e)} \stackrel{\varphi_{i-1}(e)}{\longrightarrow} \cdots\cdots \longrightarrow R^{b_0(e)} \stackrel{\varphi_0(e)}{\longrightarrow} F^e_* M \longrightarrow 0
\]
of $F^e_*M$. It follows from Lemma \ref{Schanuel} that
\[
\chi_i^F(M,\gamma) = \lim_{e \to \infty} \frac{\mu_R(\im(\varphi_i(e)))+\sum_{j =
  1}^i(-1)^jb_{i-j}(e)}{p^{e\gamma}}
\]
is independent of the choices of resolutions for $F^e_*M$. When $\gamma=\gamma(M)$, we omit $\gamma$ from the notation, and call $\chi_i^F(M)$ the {\it $i$-th (global) Frobenius Euler characteristic of $M$}. At the moment, we are not claiming that the limit exists.
\end{Remark}
%Here $\Omega_j(e)$ is the $j$-th syzygy module of any minimal free resolution of $F^e_*M$.

To study global Frobenius Betti numbers, we need a version of Lemma \ref{Schanuel} that compares minimal number of generators of the modules $\im(\varphi_i(e))$ in different resolutions. First, we record the following special case of Lemma \ref{Schanuel}.

\begin{Lemma} \label{cor:Schanuel}
Let $R$ be a Noetherian ring of Krull dimension $d$, and let $M$ be a finitely generated $R$-module. Let $0 \to \Omega \to R^n \to M \to 0$ and $0 \to \Omega' \to R^n \to M \to 0$ be short exact sequences. Then $\left| \mu_R(\Omega)-\mu_R(\Omega') \right| \leq d$.
\end{Lemma}

\begin{Definition}
Let $M$ be a finitely generated $R$-module, and let $e >0$ be an integer. Consider a free resolution of $F^e_* M$
\[
\cdots \longrightarrow  R^{b_i(e)} \stackrel{\varphi_i(e)}{\longrightarrow} R^{b_{i-1}(e)} \stackrel{\varphi_{i-1}(e)}{\longrightarrow} \cdots\cdots \longrightarrow R^{b_0(e)} \stackrel{\varphi_0(e)}{\longrightarrow} F^e_* M \longrightarrow 0.
\]
We say that the resolution is {\it minimal} if, setting $\Omega_i(e):= \im(\varphi_i(e))$, we have $\mu_R(\Omega_i(e)) = b_i(e)$ for all $i \geq 0$.
\end{Definition}

\begin{Lemma} \label{Lemma mu syzygies} Let $R$ be a Noetherian ring of Krull dimension $d$, and $M$ a finitely generated $R$-module. Let
\[
\cdots \longrightarrow  R^{b_i(e)} \stackrel{\varphi_i(e)}{\longrightarrow} R^{b_{i-1}(e)} \stackrel{\varphi_{i-1}(e)}{\longrightarrow} \cdots\cdots \longrightarrow R^{b_0(e)} \stackrel{\varphi_0(e)}{\longrightarrow} F^e_* M \longrightarrow 0.
\]
and
\[
\cdots \longrightarrow  R^{b_i'(e)} \stackrel{\psi_i(e)}{\longrightarrow} R^{b_{i-1}'(e)} \stackrel{\psi_{i-1}(e)}{\longrightarrow} \cdots\cdots \longrightarrow R^{b_0'(e)} \stackrel{\psi_0(e)}{\longrightarrow} F^e_* M \longrightarrow 0
\]
be minimal free resolutions of $F^e_*M$. %Set $\Omega_i(e) = \im(\varphi_i(e))$ and $\Omega_i'(e) = \im(\psi_i(e))$. 
Then $\left| b_i(e) - b_i'(e)\right| \leq d2^{i-1}$.
\end{Lemma}
\begin{proof} This follows immediately from a repeated application of Lemma \ref{cor:Schanuel}.
\end{proof}
%Because of Remark \ref{Remark mu syzygies}, we can make the following Definition.
%\begin{Definition}  Let $R$ be an F-finite ring, and let $M$ be a finitely generated $R$-module. We define the {\it $i$-th global Frobenius characteristic} as

%\end{Definition}

%Because of Lemma~\ref{Schanuel} and Remark~\ref{Remark mu syzygies}, we can
%equivalently define $\chi_i^F(M)$ as 
%\[
%\chi_i^F(M) = \lim_{e \to \infty} \frac{\mu_R(\Omega_i(e))+\sum_{j =
%  1}^i(-1)^jb_{i-j}(e)}{p^{e\gamma(R)}}
%\]
%provided the limit exists. Here 
%\[
%0 \lra \Omega_i(e) \lra R^{b_{i-1}(e)}
%{\longrightarrow} \dotsb\dotsb 
%\longrightarrow R^{b_0(e)} {\longrightarrow} 
%F^e_* M \longrightarrow 0
%\]
%is any truncated free resolution of $F^e_* M$ with the $i$-th syzygy. 

%\begin{Remark}
%Maybe one can consider all ''minimal'' projective resolutions instead of just ''minimal'' free in Definition \ref{Defn global Frobenius betti}. Also, I am not sure whether we can define just global Frobenius Betti numbers as $\lim_e \mu(\Omega_i(e))/p^{e\gamma(R)}$; in other words, I don't see how to prove that the limit exists.
%\end{Remark}
\begin{Remark} \label{Remark independence beta}
In the notation of Lemma \ref{Lemma mu syzygies}, let $\gamma\geq \min\{1,d\}$ be an integer. We have that
\[
\ds \beta_i^F(M,\gamma) = \lim_{e \to \infty} \frac{\mu_R(\Omega_i(e))}{p^{e\gamma}} = \lim_{e \to \infty} \frac{\mu_R(\Omega_i'(e))}{p^{e\gamma}},
\]
and is therefore independent of the choice of a minimal free resolution for $F^e_*M$. When $\gamma=\gamma(M)$, we simply write $\beta_i^F(M)$, and we call it the {\it $i$-th (global) Frobenius Betti number of $M$}. As in Remark \ref{Remark independence chi}, we are not yet claiming that the limits exist. We are only stating that one limit exists if and only if the other one does and, in this case, they coincide. Observe that $\beta_0^F(M,\gamma(R)) = \eHK(M)$ is the global Hilbert-Kunz multiplicity of $M$, therefore we know the limit exists in this case \cite{DSPY}.
\end{Remark}

\begin{Remark} \label{remark chi and chitilde}
For a finitely generated $R$-module $M$ and integers $i \geq 0$ and $\gamma$ between $\gamma(M)$ and $\gamma(R)$, recall the notation $Z_{M,\gamma}=\{ P \in \Spec(R) \mid \gamma(M_P) = \gamma\}$ introduced at the end of Section \ref{Section Uniform convergence Frob Betti}. We have already observed that $\chi_i^F(M_P,\gamma) = \chi_i^F(M_P)$ if $P \in Z_{M,\gamma}$, while $\chi_i^F(M_P,\gamma) = 0$ if $P \notin Z_{M,\gamma}$.
\end{Remark}
\begin{Proposition} \label{uniform limit chi} Let $R$ be an F-finite ring, $M$ be a finitely generated $R$-module, $i \geq 0$ and $\gamma$ be integers, with $\gamma\geq \gamma(M)$. For all integers $e \geq 0$, let $P_e \in \Spec(R)$ be such that $\chi_i(e,P_e,M) = \max\{\chi_i(e,P,M) \mid P \in \Spec(R)\}$. Then
\[
\ds \lim_{e \to \infty} \frac{\chi_i(e,P_e,M)}{p^{e\gamma}} = \lim_{e \to \infty} \chi_i^F(M_{P_e},\gamma) = \max\{\chi_i^F(M_P,\gamma) \mid P \in \Spec(R)\}.
\]
Let $\chi$ be the common value of the equation above. If either $Z_{M,\gamma} = \Spec(R)$ or $\chi \ne 0$, we also have
\[
\ds \ds \lim\limits_{e \to \infty} \frac{\chi_i(e,P_e,M)}{p^{e\gamma}} = \max\{\chi_i^F(M_P) \mid P \in Z_{M,\gamma}\}.
\]
\end{Proposition}
\begin{proof}
%First,it follows from Remark \ref{remark chi and chitilde} that
%\[
%\ds %\max\{\chi_i^F(M_P) \mid P \in Z_{M,\gamma}\} = 
%\max\{\chi_i^F(M_P,\gamma) \mid P \in Z_{M,\gamma}\} = \max\{\chi_i^F(M_P,\gamma) \mid P \in \Spec(R)\}.
%\]
Let $Q \in \Spec(R)$ be such that $\chi_i^F(M_Q,\gamma) = \max\{\chi_i^F(M_P,\gamma) \mid P \in \Spec(R)\}$, and let $\varepsilon > 0$. By Theorem \ref{uniform convergence general}, the sequence $\chi_i(e,\bullet,M)/p^{e\gamma}$ converges uniformly to its limit $\chi_i^F(M_\bullet,\gamma)$ on $\Spec(R)$. Therefore, there exists $e_0$ such that for all $e > e_0$ 
\[
\ds \left| \frac{\chi_i(e,P,M)}{p^{e\gamma}} - \chi_i^F(M_P,\gamma)\right| < \frac{\varepsilon}{2}
\]
holds for all $P \in \Spec(R)$. Then, for all $e>e_0$ we obtain
\[
\ds \chi_i^F(M_Q,\gamma) \geq \chi_i^F(M_{P_e},\gamma) > \frac{\chi_i(e,P_e,M)}{p^{e\gamma}}- \frac{\varepsilon}{2} \geq \frac{\chi_i(e,Q,M)}{p^{e\gamma}} - \frac{\varepsilon}{2} > \chi_i^F(M_Q,\gamma) - \varepsilon.
\]
Since $\varepsilon$ is arbitrary, this completes the proof of the first part of the Proposition. The second claim is now clear if $Z_{M,\gamma} = \Spec(R)$, since in this case $\chi_i^F(M_P,\gamma) = \chi_i^F(M_P)$ for all $P \in \Spec(R)$. On the other hand, if $\chi \ne 0$, by the first part there exists $P \in \Spec(R)$ such that $\chi_i^F(M_P,\gamma) = \chi \ne 0$. It then follows from Remark \ref{remark chi and chitilde} that $\max\{\chi_i^F(M_P,\gamma) \mid P \in \Spec(R)\} = \max\{\chi_i^F(M_P,\gamma) \mid P \in Z_{M,\gamma}\}$. Using that $\chi_i^F(M_P,\gamma) = \chi_i^F(M_P)$ for $P \in Z_{M,\gamma}$, we finally conclude that 
\begin{align*}
\ds \lim_{e \to \infty} \frac{\chi_i^F(e,P_e,M)}{p^{e\gamma}} & = \max\{\chi_i^F(M_P,\gamma) \mid P \in \Spec(R)\} \\
& = \max\{\chi_i^F(M_P,\gamma) \mid P \in Z_{M,\gamma}\} \\
& =  \max\{\chi_i^F(M_P) \mid P \in Z_{M,\gamma}\}. \qedhere
\end{align*}
\end{proof}
The assumptions for the second claim in the Proposition are needed, as the following example shows.
\begin{Example} Let $S=\F_p$, and $T=\F_p(t)$, and consider the ring $R=S \times T$. Since $\chi_1(e,S \times 0,R) = -p^e$ and $\chi_1(e,0 \times T,R) = -1$, using the notation of Proposition \ref{uniform limit chi} we have $P_e=0\times T$ for all $e$. Using  $\gamma=\gamma(R)=1$, it then follows that $\lim\limits_{e \to \infty} \frac{\chi_1(e,P_e,R)}{p^e} = 0$. However, one has $\max\{\chi_1^F(R_P) \mid P \in Z_{R,1}\} = \chi_1^F(R_{S\times 0}) = \chi_1^F(\F_p(t)) = -1$. Observe that there is no contradiction with the first part of the Proposition, since $\max\{\chi_1^F(R_P,1) \mid P \in \Spec(R)\} = \chi_1^F(R_{0 \times T},1) = \chi_1^F(\F_p,1)= 0$.
\end{Example}
\begin{Theorem} \label{global betas}
Let $R$ be an F-finite ring of prime characteristic $p>0$, $M$ be a finitely generated $R$-module, and $\gamma \geq \gamma(M)$ be an integer such that $\gamma \geq \min\{1,\dim(R)\}$. For every $e \in \Z_{>0}$, fix any free resolution $(G_\bullet(e),\varphi_\bullet(e))$ of the module $F^e_*M$:
\[
\xymatrix{
\cdots \ar[r] & R^{\b_i(e)} \ar[r]^-{\varphi_i(e)} &  R^{b_{i-1}(e)} \ar[r]^-{\varphi_{i-1}(e)} & \cdots \ar[r] & R^{b_0(e)} \ar[r]^-{\varphi_0(e)} & F^e_*M \ar[r] & 0.
}
\]
For $i \in \Z_{\geq 0}$, let $\Omega_i(e) = \im(\varphi_{i}(e))$. Then:
\begin{enumerate}
\item The limit $\ds \chi_i^F(M,\gamma)=\lim\limits_{e \to \infty} \frac{\mu_R(\Omega_i(e)) + \sum_{j=1}^i (-1)^j b_{i-j}(e)}{p^{e\gamma}}$ exists. 
\item $\ds \chi_i^F(M,\gamma) = \max\{\chi_i^F(M_P,\gamma) \mid P \in \Spec(R) \}$. Moreover, if either this value is non-zero or $Z_{M,\gamma}=\Spec(R)$, then it is also equal to $\max\{\chi_i^F(M_P) \mid P \in Z_{M,\gamma}\}$.
\item Assume further that, for all $e \in \Z_{>0}$, the free resolution $G_\bullet(e)$ is chosen to be minimal. Then the limit $\ds \beta_i^F(M,\gamma) = \lim\limits_{e\to \infty}\frac{b_i(e)}{p^{e\gamma}}$ exists.
\end{enumerate}
\end{Theorem}
\begin{proof} For all $P \in \Spec(R)$ and $e \in \Z_{>0}$, localizing the resolution $(G_\bullet(e),\varphi_\bullet(e))$ at $P$ gives an exact sequence:
\[
0 \longrightarrow \Omega_i(e)_P \longrightarrow R_P^{b_{i-1}(e)} \longrightarrow \dotsb\dotsb \longrightarrow R_P^{b_0(e)} \longrightarrow F^e_* (M_P) \longrightarrow 0,
\]
which gives
\[
\mu_{R_P}(\Omega_i(e)_P) + \sum_{j=1}^{i} (-1)^{j} b_{i-j}(e) = \chi_i(e,P,M).
\]
In particular, this shows that 
\[
\ds \max\{\mu_{R_P}(\Omega_i(e)_P) \mid P \in \Spec(R)\} 
+ \sum_{j=1}^{i} (-1)^{j} b_{i-j}(e)
= \max\{\chi_i(e,P,M) \mid P \in \Spec(R)\}.
\]
For all $e \in \Z_{>0}$, let $P_e$ be a prime that achieves such
maximum. Then, by the Forster-Swan Theorem \cite{Forster,Swan}, we have that
$\mu_{R_{P_e}}(\Omega_i(e)_{P_e}) \leq \mu_R(\Omega_i(e))
\leq \mu_{R_{P_e}}(\Omega_i(e)_{P_e}) + \dim(R)$. Therefore, for all
$e>0$, we have 
\[
\frac{\chi_i(e,P_e,M)}{p^{e\gamma}} 
\leq \frac{\mu_R(\Omega_i(e)) +\sum_{j=0}^i (-1)^{j} b_{i-j}(e)}{p^{e\gamma}} 
\leq \frac{\chi_i(e,P_e,M) +  \dim(R)}{p^{e\gamma}}. 
\]
Part (1) and (2) now follow from Proposition \ref{uniform limit chi}, since the difference between the two terms on the sides of the inequality goes to zero because of our assumptions on $\gamma$. Given that $\chi_i^F(M,\gamma)$ exists as a limit, for part (3) it is enough to observe that, if $G_\bullet$ is minimal, then we have
\[
\ds \beta_i^F(M,\gamma) = \chi_i^F(M,\gamma) + \chi_{i-1}^F(M,\gamma). \qedhere
\]
\end{proof}

As a consequence of Theorem~\ref{global betas}~(2), we have that $\chi_i^F(M,\gamma)=0$ for all $i \in \Z_{\geq 0}$ whenever $\gamma>\gamma(M)$. Therefore, $\beta_i^F(M,\gamma) = 0$ for $\gamma> \gamma(M)$ as well.

Unlike the case of $\chi_i^F(M,\gamma)$, we do not know whether $\beta_i^F(M,\gamma)$ coincides with the maximal value of the local invariants achieved on $\Spec(R)$. We ask it here as a question.
\begin{Question} \label{question_beta_local_max} Does the following equality hold:
\[
\ds \beta_i^F(M,\gamma) = \max\{\beta_i^F(M_P,\gamma) \mid P \in \Spec(R)\}?
\]
\end{Question}
The following example provides evidence that studying Question \ref{question_beta_local_max} may lead to some interesting consequences.
%\begin{Remark} Note that $\chi_0^F(M) = e_{HK}(M)$. Therefore Theorem \ref{global Euler = maximum local} recovers Theorem \ref{HK exists} and Theorem \ref{What is HK}.
% \end{Remark}
\begin{Example}
Let $Q$ be an F-finite regular ring, and let $f$ be an non-unit element of $Q$. Let $R=Q/(f)$, and assume that $Z_{R,\gamma(R)}= \Spec(R)$. By \cite[Example 3.2]{DSHNB}, for all $P \in \Spec(R)$, we have 
\begin{eqnarray*}
\beta_i^F(R_P) = \left\{
\begin{array}{ll} \eHK(R_P) & i=0 \\ \eHK(R_P)-\s(R_P) & i > 0
\end{array}
\right.
\end{eqnarray*}
where $\s(R_P)$ is the F-signature of the local ring $R_P$. Therefore
\begin{eqnarray*}
\chi_i^F(R_P) = \left\{
\begin{array}{ll} \eHK(R_P) & i \mbox{ even} \\ -\s(R_P) & i \mbox{ odd}
\end{array}
\right.
\end{eqnarray*}
By Theorem \ref{global betas} we have that
\[
\ds \chi_1^F(R) = \max\{\chi_1^F(R_P) \mid P \in \Spec(R)\} = -\min\{\s(R_P) \mid P \in \Spec(R)\},
\]
and
\[
\ds \chi_0^F(R) = \max\{\chi_0^F(R_P) \mid \in \Spec(R)\} = \max\{\eHK(R_P) \mid P \in \Spec(R)\}.
\]
Since $\beta_1^F(R) = \chi_1^F(R) + \chi_0^F(R)$, it follows that $\beta_1^F(R) = \max\{\beta_1^F(R_P) \mid P \in \Spec(R)\} = \max\{\eHK(R_P) - \s(R_P) \mid P \in \Spec(R)\}$ if and only if 
\[
\ds \{P \in \Spec(R) \mid \eHK(R_P) \mbox{ is maximal}\} \cap \{P \in \Spec(R) \mid \s(R_P) \mbox{ is minimal}\} \ne \emptyset.
\]
\end{Example} 
%\begin{Remark}
%Unfortunately, we are not able to explicitly compute the Frobenius Betti numbers in many other cases....
%\end{Remark}
The following result extends \cite[Corollary 3.2]{AberbachLi} to the global setting.
\begin{Theorem} \label{THM beta regular} Let $R$ be an F-finite ring such that $Z_{R,\gamma(R)}= \Spec(R)$. Then $\beta_i^F(R) = 0$ for some (equivalently, for all) $i>0$ if and only if $R$ is regular.
\end{Theorem}
\begin{proof} Assume that $\beta_i^F(R) = 0$ for some $i>0$. If $\Omega_i(e)$ is the $i$-th syzygy module of any minimal free resolution of $F^e_*R$, we always have $\beta_i(e,P,R) \leq \mu_{R_P}(\Omega_i(e)_P) \leq \mu_R(\Omega_i(e))$. Since $\beta_i^F(R) = \lim\limits_{e \to \infty} \frac{\mu_R(\Omega_i(e))}{p^{e\gamma(R)}} = 0$, we have $\beta_i^F(R_P) = 0$ for all $P \in \Spec(R)$. By \cite{AberbachLi}, we conclude that $R_P$ is regular for all primes $P$, hence, $R$ is regular. Conversely, if $R$ is regular, for all $P \in \Spec(R)$ we have that $\beta_i^F(R_P) = 0$ for all $i >0$, and $\eHK(R_P) = \beta_0^F(R_P) =1$. In particular, $\chi_i^F(R_P) = (-1)^i$ for all $P \in \Spec(R)$. By Theorem~\ref{global betas}, we have $\chi_i^F(R) = (-1)^i$ for all $i$, and it follows that $\beta_i^F(R) = \chi_i^F(R) + \chi_{i-1}^F(R) = 0$.
\end{proof}

\begin{comment}
$\chi_i^F(R_P) = (-1)^i$ for all $P \in \Spec(R)$. As before, let $\Omega_i(e)$ be the $i$-th syzygy module of any minimal free resolution of $F^e_*R$. By the Forster-Swan theorem \cite{Forster,Swan}, there exists $P \in \Spec(R)$ such that $\mu_R(\Omega_i(e)) \leq \mu_{R_P}(\Omega_i(e)_P) + d$, where $d = \dim(R)$. Localizing the minimal free resolution of $F^e_*R$ at $P$ and repeatedly using localizations of short exact sequences of syzygies we obtain that
\[
\ds \mu_R(\Omega_i(e)) \leq \mu_{R_P}(\Omega_i(e)_P) + d = \beta_i(e,P,R) + \mu_R(\Omega_{i-1}(e)) - \mu_{R_P}(\Omega_{i-1}(e)_P) + d.
\]
Now choose $Q \in \Spec(R)$ such that $\mu_R(\Omega_{i-1}(e)) \leq \mu_{R_Q}(\Omega_{i-1}(e)_Q) + d$. In this way, we obtain
\[
\ds \mu_R(\Omega_i(e))  \leq \beta_i(e,P,R) + \mu_{R_Q}(\Omega_{i-1}(e)_Q) - \mu_{R_P}(\Omega_{i-1}(e)_P) + 2d.
\]
Note that, inductively, we get
\begin{align*}
\ds \mu_{R_Q}(\Omega_{i-1}(e)_Q) - \mu_{R_P}(\Omega_{i-1}(e)_P)  & = \beta_{i-1}(e,Q,R)  -\beta_{i-1}(e,P,R) - \left(\mu_{R_Q}(\Omega_{i-2}(e)_Q) - \mu_{R_P}(\Omega_{i-2}(e)_P)\right) \\
& = \chi_{i-1}(e,Q,R) - \chi_{i-1}(e,P,R).
\end{align*}
Therefore we have
\[
\ds \mu_R(\Omega_i(e)) \leq \beta_i(e,P,R) + \chi_{i-1}(e,Q,R) - \chi_{i-1}(e,P,R) + 2d = \chi_i(e,P,R) + \chi_{i-1}(e,Q,R) + 2d.
\]
Dividing by $p^{e\gamma(R)}$ and taking $\lim_{e \to \infty}$ we obtain
\[
\ds \beta_i^F(R) = \lim_{e \to \infty} \frac{\mu_R(\Omega_i(e))}{p^{e\gamma(R)}} \leq \chi_i^F(R_P) + \chi_{i-1}^F(R_Q) =  0.
\]
\end{comment}
Given a finitely $R$-module $M$, we let $\Assh(M,\gamma)$ denote the set of associated primes $P$ of $M$ such that $\gamma(R/P) = \gamma$. We now establish the behavior of the invariants $\chi_i^F(-,\gamma)$ under short exact sequences. As a consequence, we extend a version of associativity formula for the global Hilbert-Kunz multiplicity \cite[Corollary 3.10]{DSPY}.
\begin{Proposition} Let $R$ be an F-finite ring, and $0 \to A \to B \to C \to 0$ be a short exact sequence of finitely generated $R$-modules, and $\gamma \geq \gamma(B)$ be an integer such that $\gamma \geq \min\{1,\dim(R)\}$. For $i \in \Z_{\geq 0}$, we have
\begin{enumerate}
\item $\chi_i^F(B,\gamma) = \chi_i^F(A \oplus C,\gamma)$.
\item $\chi_i^F(B,\gamma) \leq \chi_i^F(A,\gamma)+\chi_i^F(C,\gamma)$.
\item $\ds \chi_i^F(B,\gamma) =\chi_i^F\left(\bigoplus_{P \in \Assh(B,\gamma)} \bigoplus_{i=1}^{\lambda(B_P)} R/P,\gamma\right)$.
\end{enumerate}
Analogous statements hold true for $\beta_i^F(-,\gamma)$ in place of $\chi_i^F(-,\gamma)$.
\end{Proposition}
\begin{proof}
We prove $(1)$. It follows from \cite[Proposition 1 (b)]{Seibert} that, for all $P \in \Spec(R)$, we have equalities $\chi_i^F(B_P,\gamma) = \chi_i^F(A_P,\gamma) + \chi_i^F(C_P,\gamma) = \chi_i^F((A\oplus C)_P,\gamma)$. Using Theorem~\ref{global betas}~(2), we conclude that
\begin{align*}
\ds \chi_i^F(B,\gamma) &= \max\{\chi_i^F(B_P,\gamma) \mid P \in \Spec(R)\} \\
& = \max\{\chi_i^F((A\oplus C)_P,\gamma) \mid P \in \Spec(R)\}  = \chi_i^F(A\oplus C,\gamma).
\end{align*}
For $(2)$, let $P \in \Spec(R)$ be such that $\chi_i^F(M,\gamma) = \chi_i^F(M_P,\gamma)$, which exists by Theorem~\ref{global betas}~(2). Using the result of Seibert mentioned above, we get
\[
\ds \chi_i^F(B,\gamma) = \chi_i^F(B_P,\gamma) =\chi_i^F(A_P,\gamma) + \chi_i^F(C_P,\gamma) \leq \chi_i^F(A,\gamma) + \chi_i^F(C,\gamma).
\]
Part $(3)$ follows from a repeated application of $(1)$, using a prime filtration of $B$.

Analogous statements for $\beta_i^F(-,\gamma)$ follow at once from the relation $\beta_i^F(-,\gamma) = \chi_i^F(-,\gamma)+\chi_{i-1}^F(-,\gamma)$.
\end{proof}

We end this section by observing that, for positively graded algebras over a local ring $(R_0,\m_0)$, the global Frobenius Betti numbers coincide with the ones in the localization at the irrelevant maximal ideal $\m=\m_0+R_{>0}$.
\begin{Proposition} \label{prop graded betti} Let $(R_0,\m_0,k)$ be an F-finite local ring and let $R$ be a positively graded algebra of finite type over $R_0$. Let $R_{>0}$ be the ideal of $R$ generated by elements of positive degree, $\m=\m_0+R_{>0}$, and $M$ be a finitely generated graded $R$-module. Let $\gamma \geq \gamma(M)$ be an integer such that $\gamma \geq \min\{1,\dim(R)\}$. For all $i \in \Z_{\geq 0}$, we have
\[
\ds \beta_i^F(M,\gamma) = \beta_i^F(M_\m,\gamma) \ \mbox{ and } \ \chi_i^F(M,\gamma) = \chi_i^F(M_\m,\gamma).
\]
\end{Proposition}
\begin{proof}
It is sufficient to show the equality $\beta_i^F(M,\gamma) = \beta_i^F(M_\m,\gamma)$. Observe that $F^e_*M$ is a $\Q$-graded $R$-module. Given any finitely generated graded $R$-module $N$, the minimal number of generators of $N$ is the length of $N/\m N$, by the graded version of Nakayama's Lemma. Applying this observation to the $\Q$-graded syzygies of $F^e_*M$ for $e \in \Z_{>0}$, one can construct a graded exact sequence 
\begin{equation}
\label{eqn}
\xymatrix{
0 \ar[r] & \Omega_i(e) \ar[r] &\ds  \bigoplus_{j=1}^{b_{i-1}(e)}R[n_{i-1,j}] \ar[r]^-{\varphi_{i-1}(e)} & \cdots\cdots \ar[r] & \ds \bigoplus_{j=1}^{b_0(e)}R[n_{0,j}] \ar[r]^-{\varphi_0(e)} & F^e_*M \ar[r] & 0,
}
\end{equation}
where each syzygy $\Omega_j(e) = \im(\varphi_{j}(e))$ is graded, and $b_j(e) = \mu_R(\Omega_j(e)) = \lambda_R(\Omega_j(e)/\m\Omega_j(e))$ for all $j$. In the resolution, $R[n_{\ell,j}]$ denotes the cyclic $\Q$-graded free module with generator in degree $-n_{\ell,j} \in \Q$. In particular, this is a minimal free resolution of $F^e_*M$, and it follows from Theorem~\ref{global betas}~$(3)$ %and Nakayama's Lemma again 
that $\beta_i^F(M,\gamma) = \lim\limits_{e \to \infty} \frac{b_i(e)}{p^{e\gamma}}$. % = \lim\limits_{e \to \infty} \frac{\lambda_R(\Omega_i(e)/\m \Omega_i(e))}{p^{e\gamma(R)}}.
On the other hand, since all the maps and all the modules in (\ref{eqn}) are graded minimal, after localizing at $\m$ we obtain a minimal free resolution of $F^e_*(M_\m)$:
\[
\xymatrix{
0 \ar[r] & \Omega_i(e)_\m \ar[r] & R_\m^{b_{i-1}(e)} \ar[r] & \cdots \ar[r] & R_\m^{b_0(e)} \ar[r] & F^e_*(M_\m) \ar[r] & 0.
}
\]
In particular, since $\lambda_{R_\m}(\Tor_i^R(k,F^e_*(M_\m))) = \lambda_{R}(\Omega_i(e)/\m\Omega_i(e)) = b_i(e)$, we have
\[
\ds \beta_i^F(M_\m,\gamma) = \lim_{e \to \infty} \frac{\lambda_{R_\m}(\Tor_i(k,F^e_*(M_\m)))}{p^{e\gamma}} = \lim_{e \to \infty} \frac{b_i(e)}{p^{e\gamma}} = \beta_i^F(M,\gamma). \qedhere
\]
\end{proof}
\begin{Corollary} Let $R$ and $\m$ be as in Proposition~\ref{prop graded betti}% and assume that $Z_R=\Spec(R)$ (for instance, $R$ is a domain)
. For all finitely generated $R$-modules $M$, we have $\eHK(M) = \eHK(M_\m)$ .
\end{Corollary}
\begin{proof}
By Proposition \ref{prop graded betti}, we have $\eHK(M) = \beta_0^F(M,\gamma(R)) = \beta_0^F(M_\m,\gamma(R))$. Since $\gamma(R) = \gamma(R_\m)$, we get $\eHK(M)  = \beta_0^F(M_\m,\gamma(R_\m)) = \eHK(M_\m)$.
\end{proof}
\section{Background on F-splitting ratio of local rings} Let $(R,\m,k)$ be an F-finite local ring of prime characteristic $p>0$. Aberbach and Enescu introduced the concepts of splitting prime and F-splitting ratio of a local F-finite ring in \cite{AberbachEnescu2005}. Assume that $R$ is F-pure, that is, the Frobenius map is pure as a map of rings. In our assumptions, this is the same as requiring that $R$ is F-split \cite[Corollary~5.3]{HochsterRoberts1976}. For a finitely generated $R$-module $M$, we let $\frk_R(M)$ be the maximal rank of a free summand of $M$. Equivalently, $\frk_R(M)$ is the maximal rank of a free module $G$ for which there is a surjection $M \to G \to 0$. For all $e \in \Z_{>0}$, we let $a_e(R) = \frk_R(F^e_*R)$ be the {\it $e$-th splitting number of $R$}. The {\it splitting dimension of $R$} is
\[
\sdim(R):=\sup\left\{\ell \in \Z_{\geq 0} \ \bigg| \ \liminf_{e\to \infty }\frac{a_e(R)}{p^{e(\ell+\alpha(\m))}}>0\right\}.
\]
The {\it F-splitting ratio of $R$} is defined to be the limit
\[
r_F(R):=\lim_{e\to \infty}\frac{a_e(R)}{p^{e(\sdim(R)+\alpha(\m))}},
\]
which always exists \cite[Theorem~4.9]{Tucker2012} and is always positive for F-pure rings by work of Blickle, Schwede, and Tucker \cite[Corollary~4.3]{BST2012}. 
\begin{Remark} %If $R$ is not F-pure then the splitting dimension of $R$ is set to be equal to $-1$, \footnote{Historically, authors have set $\sr(R)=-\infty$ when $R$ is not F-pure, but it will advantageous in this article to set the splitting rate to be $-1$ in such a scenario.} and the F-splitting ratio is $0$. Moreover, 
Observe that, when $\sdim(R) = \dim(R)$, the F-splitting ratio is equal to the F-signature of $R$. 
\end{Remark}

Continue to let $(R,\m,k)$ denote an F-finite and F-pure local ring of prime characteristic $p>0$. For each $e\in\Z_{>0}$ let $I_e=\{r\in R\mid R\xrightarrow{\cdot F^e_*(r)}F^e_*R\mbox{ is not pure}\}$ be the $e$-th splitting ideal of $R$. Aberbach and Enescu show in \cite{AberbachEnescu2005} that $\PP:=\bigcap_{e\in \Z_{>0}}I_e$ is a prime ideal of $R$ and $R/\PP$ is a strongly F-regular local ring. The ideal $\PP$ is called the splitting prime of the local ring $R$. Moreover, it is shown in \cite{BST2012} that the splitting dimension of $R$ is the Krull dimension of the local ring $R/\PP$.

We recall that a graded $\F_p$-subalgebra $\D$ of $\bigoplus_{e \in \Z_{\geq 0}} \Hom_R(F^e_*R,R)$, with $\D_0=\Hom_R(R,R)$ and multiplication $\varphi \bullet \psi = \varphi \circ F^e_*\psi \in \D_{e+e'}$ for all $\varphi \in \D_e$ and $\psi \in \D_{e'}$, is called a Cartier algebra. If $\D_e = \Hom_R(F^e_*,R)$ for all $e$, we refer to $\D$ as the full Cartier algebra on $R$. See \cite{BST2012} for more details on Cartier algebras.

If $I\subseteq R$ is an ideal, then we let $\D_{R/I}$ be the Cartier algebra on $R/I$ whose $e$-th graded component is denoted by $\D_{R/I,e}$ and consists of $R/I$-linear maps $\varphi: F^e_*(R/I)\to R/I$ which can be factored through an $R$-linear map $\phi: F^e_*R\to R$. That is, there exists commutative diagram of $R$-modules of the form
\[
\xymatrix{
F^e_*(R/I) \ar[r]^\varphi & R/I \\ 
F^e_*R \ar[u] \ar@{-->}^-{\exists \phi}[r] & R \ar[u]
}
\]
Observe that the construction of this Cartier algebra did not require $R$ to be local. Moreover, if $P$ is a prime ideal of $R$ which contains $I$, then the localized Cartier algebra $(\D_{R/I})_P$ agrees with $\D_{R_P/IR_P}$. 

We now recall the definition of splitting numbers of a pair $(R,\D)$ in the local case. Let $(R,\m,k)$ be a local F-finite and F-pure ring of prime characteristic $p>0$, and $\D$ be a Cartier algebra. We let $a_e(R,\D)$ be the largest rank of a free $\D$-summand of $F^e_*R$. More explicitly, we look at the largest rank of a free $R$-module $G \cong \bigoplus R$ for which there is a surjection $F^e_*R \xrightarrow{\varphi} G \to 0$, with $\varphi$ that is a direct sum of elements in $\D_e$ when viewed as an element of $\Hom_R(F^e_*R,G) \cong \bigoplus \Hom_R(F^e_*R,R)$. It was proved in \cite{BST2012} that, if $\D$ is the full Cartier algebra on $R$, and $\PP$ is the splitting prime of $R$, one has
\[
a_e(R)=a_e(R/\PP, \D_{R/\PP}).
\]
%It follows that $r_F(R)=\s(R/\PP,\D_{R/\PP})$. 
We record the following theorem of Blickle, Schwede, and Tucker for future reference.

\begin{Theorem}[\cite{BST2012}]\label{Splitting ratio is F-signature of a Cartier algebra} Let $(R,\m,k)$ be a local F-finite and F-pure ring of prime characteristic $p>0$. Let $\D$ be the full Cartier algebra on $R$, and $\mathcal{P}$ be the splitting prime of $R$. Then $a_e(R) = a_e(R/\PP,\D_{R/\PP})$ for all $e \in \Z_{>0}$, and thus $r_F(R) = \s(R/\PP,\D_{R/\PP}) = r_F(R/\PP,\D_{R/\PP})$. In particular, the F-splitting ratio of $R$ is strictly positive.
\end{Theorem}

\section{Global F-splitting Ratio} \label{Section global r_F}
Let $R$ be an F-finite ring of prime characteristic $p>0$, but not necessarily local and let $M$ be a finitely generated $R$-module. For $e \in \Z_{>0}$ we let $a_e(M) = \frk_R(F^e_*M)$, and assume that $a_e(M)>0$ for some $e$.  Under these assumptions, we define the {\it F-splitting rate of $M$} to be
\[
\sr(M):=\sup\left\{\ell \in \Z_{\geq 0} \ \bigg| \ \liminf_{e\to \infty }\frac{a_e(M)}{p^{e\ell}}>0\right\}.
\]
If $(R,\m,k)$ is local, then $\sr(R)=\sdim(R)+\alpha(\m)$. Moreover, if $\PP$ the splitting prime of $(R,\m,k)$, then $\sr(R)=\gamma(R/\PP)$. We define the {\it global F-splitting ratio of $M$} to be
\[
r_F(M)=\lim_{e\to \infty }\frac{a_e(M)}{p^{e\sr(M)}},
\]
provided the limit exists. When $a_e(M)=0$ for all $e\in \Z_{>0}$ we set $\sr(M)=-1$ and $r_F(M)=0$. The main purpose of this section is to prove existence of the global F-splitting ratio of a finitely generated module $M$. To do so we first must better understand the global behavior of the numbers $a_e(R_P)$, as $P$ varies in $\Spec(R)$. We then develop a local theory of F-splitting ratio of finitely generated $R$-modules. We then can use the global theory of the splitting numbers $a_e(R_P)$ to understand the global theory of the splitting numbers $a_e(M_P)$ and then we invoke results of \cite{DSPY2} to prove the existence of global F-splitting ratio of a finitely generated module.

In order to develop the theory of splitting ratios we must first discuss and understand properties of centers of F-purity, i.e., compatibly split subvarieties, whose properties are developed by Schwede in \cite{Schwede2009} and \cite{Schwede2010}, and by Kumar and Mehta in \cite{KumarMehta}.

Let $R$ be an F-finite and F-pure ring of prime characteristic $p>0$. A prime ideal $P\in \Spec(R)$ is called a center of F-purity if for every $x\in P$ and every $e \in \Z_{>0}$ the map 
\[
R_P\xrightarrow{\cdot F^e_*x}F^e_*(R_P)
\] 
is not pure as a map of $R_P$-modules. If $R$ is local and $\PP$ the splitting prime of $R$ then $\PP$ is the unique maximal center of F-purity of $R$, \cite[Remark~4.4]{Schwede2010}. An important property enjoyed by all F-pure rings is that they only admit finitely many centers of F-purity.

\begin{Theorem}[{\cite[Theorem~C]{Schwede2009}}, {\cite[Theorem~1.1]{KumarMehta}}]\label{Finitely many centers of F-purity} Let $R$ be an F-finite and F-pure ring of prime characteristic $p>0$. Then $R$ admits only finitely many centers of F-purity.
\end{Theorem}

Also crucial to our proof of existence of global F-splitting ratio will be that Cartier algebras of the form $\D_{R/I}$ described above satisfy the following technical condition.

\begin{Condition} \label{condition *} Let $R$ be an F-finite ring and $\D$ a Cartier algebra. We say that $\D$ satisfies condition $(*)$ if we require that for each $\varphi\in \D_{e+1}$ that the natural map $i\circ \varphi\in \D_{e}$ where $i:F^e_*R\to F^{e+1}_*R$ is the Frobenius. 
\end{Condition}

\begin{Lemma}\label{Class of Cartier algebras satisfying condition (*)} Let $R$ be an F-finite ring of prime characteristic $p>0$ and $I\subseteq R$ be an ideal. Assume that the Cartier algebra $\D$ on $R$ satisfies $(*)$. Then the Cartier algebra $\D_{R/I}$ on $R/I$ satisfies condition $(*)$ as well.
\end{Lemma}

\begin{proof}
Let $\varphi \in \D_{R/I,e+1}$, and $i:F^e_*(R/I)\to F^{e+1}_*(R/I)$  be the Frobenius map on $F^e_*(R/I)$. We are assuming there exists a commutative diagram of $R$-modules of the form
\[
\xymatrix{
F^{e}_*(R/I)\ar[r]^-i &F^{e+1}_*(R/I)\ar[r]^-\varphi &R/I\\
F^{e}_*R \ar[u] \ar@{-->}[r] & F^{e+1}_*R \ar[u] \ar[r]^-\phi & R \ar[u]
}
\]
The Frobenius map on $F^e_*(R/I)$ can be lifted by the Frobenius map on $F^e_*R$. Therefore the above commutative diagram can be filled in, and it follows that $\varphi\circ i\in \D_{R/I,e}$.
\end{proof}

We are almost ready to prove a uniform bound result for the localized splitting numbers $a_e(R_P)$ of an F-finite ring $R$, but first we recall a uniform bound result of \cite{Polstra2018}. We use the following notation: as in Section \ref{Background}, given a prime $P \in \Spec(R)$ we let $\alpha(P) = \log_p[F_*\kappa(P):\kappa(P)]$ and $\gamma(R) = \max\{\alpha(P) \mid P \in \Spec(R)\}$. Moreover, given a pair $(R,\D)$, $P \in \Spec(R)$ and $e \in \Z_{>0}$, we let $a_e(R_P,\D_P)$ be the maximal rank of a free $\D_P$-summand of $F^e_*(R_P)$. In the case when $\D = \Hom_R(F^e_*R,R)$ is the full Cartier algebra, we simply write $a_e(R_P)$, which is also equal to $\frk_{R_P}(F^e_*R_P)$.

\begin{Theorem}[{\cite{Polstra2018}}]\label{Uniform convergence for Cartier algebras satisfying condition (*)} Let $R$ be an F-finite ring, and $\D$ be a Cartier algebra satisfying condition $(*)$. There exists a constant $C$ such that for all $P\in \Spec(R)$ and all $e\in \Z_{>0}$
\[
\left|a_e(R_P,\D_P)-p^{e\gamma(R_P)}\s(R_P,\D_P)\right|\leq Cp^{e(\gamma(R)-1)}.
\]
\end{Theorem}

Using this, we obtain uniform bounds for the difference of localized splitting numbers of an F-finite F-pure ring and the corresponding F-splitting ratios. We will obtain more general results in Theorem~\ref{global r(M)}.
\begin{Theorem}\label{Main Theorem about global F-splitting ratio}
Let $R$ be an F-finite ring and F-pure ring. There is a constant $C\in \R$ such that for all $P\in\Spec(R)$ and $e\in \Z_{>0}$
\[
\left|a_e(R_P)-p^{e\sr(R_P)}r_F(R_P)\right|\leq Cp^{e(\sr(R_P)-1)}. 
\]
\end{Theorem}
\begin{proof}
Let $Y=\{\p_1,\ldots,\p_N\}$ be the finitely many centers of F-purity of $\Spec(R)$, and $\D$ be the full Cartier algebra on $R$. Observe that $\D$ trivially satisfies condition $(*)$. For each $\p_i$, let $C_i$ be a constant as in Theorem~\ref{Uniform convergence for Cartier algebras satisfying condition (*)} for the pair $(R/\p_i,\D_{R/\p_i})$. We claim that we can choose $C=\max\{C_1,\ldots,C_N\}$. In fact, given $P \in \Spec(R)$, there is a unique $\p_i \in Y$ such that $\p_iR_P$ is the splitting prime of $R_P$. If we let $S=R/\p_i$, by Theorem \ref{Splitting ratio is F-signature of a Cartier algebra} we have that $a_e(R_P) = a_e(S_P,\D_{S_P})$ and $r_F(R_P) = r_F(S_P,\D_{S_P})$. As the Cartier algebra $\D_S$ still satisfies conition $(*)$, it then follows from Theorem~\ref{Uniform convergence for Cartier algebras satisfying condition (*)} that
\begin{align*}
\ds \left|a_e(R_P)-p^{e\sr(R_P)}r_F(R_P)\right| & = \left|a_e(S_P,\D_{S_P})-p^{e\sr(R_P)}r_F(S_P,\D_{S_P})\right| \\
& \leq C_ip^{e(\gamma(S_P)-1)} \leq C p^{e(\sr(R_P)-1)}. \qedhere
\end{align*}
\end{proof}

A consequence of Theorem~\ref{Main Theorem about global F-splitting ratio} is the following:

\begin{Corollary}\label{Corollary to Main Theorem about F-splitting ratio} Let $R$ be an F-finite and F-pure ring of prime characteristic $p>0$. Then the normalized splitting number functions $\tilde{a}_e: \Spec(R)\to \R$ mapping $P\mapsto a_e(R_P)/p^{e\sr(R_P)}$ converge uniformly as $e\to \infty$ to the F-splitting ratio function $r_F:\Spec(R)\to \R$ mapping $P\mapsto r_F(R_P)$.
\end{Corollary}

\subsection{Splitting ratios for modules over a local ring and uniform bounds}

The theory of splitting ratios over a local ring developed in \cite{AberbachEnescu2005} and \cite{BST2012} only concerns itself with the Frobenius splitting numbers $a_e(R)$ of a local ring $(R,\m,k)$. In this subsection we extend the local theory by studying the Frobenius splitting numbers of finitely generated modules. We begin by observing if $(R,\m,k)$ is local and $M$ is finitely generated then the splitting rate of $M$ defined above is either $-1$ or $\sr(R)$.

\begin{Lemma}\label{splitting ratio of a finitely generated module} Let $(R,\m,k)$ be a local F-finite ring of prime characteristic $p>0$ and let $M$ be a finitely generated $R$-module. If $a_{e_0}(M)>0$ for some $e_{0}\in \Z_{>0}$ then $\sr(M)=\sr(R)$.
\end{Lemma}

\begin{proof} Choose an onto $R$-linear map $R^{\oplus n}\to M$. Then $a_e(M)\leq na_e(R)$ and it follows that $\sr(M)\leq \sr(R)$. If $F^{e_0}_*M\cong R\oplus M_{e_0}$ for some $e_0$ then $F^{e+e_0}_*M\cong F^{e}_*R\oplus F^e_*M_{e_0}$ for each $e\in \Z_{>0}$. Therefore $a_e(R)\leq a_{e+e_0}(M)$ for each $e\in\Z_{>0}$ and $\sr(R)\leq \sr(M)$.
\end{proof}

Let $R$ be an F-finite ring and $M$ a finitely generated $R$-module. The \emph{F-split locus of $M$} is $\fsl(M) = \{P \in \Spec(R) \mid F^e_*(M_P) \text{ has a free summand for some } e > 0\}$. Observe that, if $F^e_*(M_P)$ has a free summand, then so does $F^e_*(R_P)$. Therefore $\fsl(M)\subseteq \fsl(R)$. Moreover, Lemma~\ref{splitting ratio of a finitely generated module} proves that, if $P\in \fsl(M)$, then the splitting rates of $M_P$ and $R_P$ agree.

 Our next lemma establishes the existence of the F-splitting ratio of a finitely generated module over a local ring $(R,\m,k)$ under the assumption that $\m$ is the splitting prime ideal of $R$.

\begin{Lemma}\label{m-center}
Let $(R, \m, k)$ be an F-finite and F-pure local ring, with $\m$
being its splitting prime. Let $\gamma = \gamma(\m)$. 
For every $e \ge 0$, write $F^{e}_*M \cong R^{\oplus a_{e}(M)} \oplus
M_{e}$. Then
\begin{enumerate}
\item The sequence $\{a_e(R)/p^{e\gamma}\}$ is the constant sequence $\{1\}$. In particular $r_F(R)=1$.
\item The sequence $\{a_e(M)/p^{e\gamma}\}_{e \geq 0}$ is a non-decreasing sequence of integers, and therefore eventually constant. In particular, the F-splitting ratio $r_F(M)$ exists. Moreover,
  $\sr(M) = \gamma \iff r_F(M) >0 \iff \m \in \fsl(M)$. 
\item If $a_{e}(M)/p^{e\gamma} = r_F(M)$ then $a_{e'}(M_e)
  = 0$ for all $e' \ge 0$.
\end{enumerate}
\end{Lemma}

\begin{proof}
As discussed before, if we let $I_e=\{r\in R\mid R\xrightarrow{\cdot F^e_*r}F^e_*R\mbox{ does not split}\}$ then $I_e$ is an $\m$-primary ideal such that $\lambda(R/I_e)=\frac{a_e(R)}{p^{e\gamma}}$, and $\bigcap_{e \in \Z_{>0}} I_e$ is the splitting prime of $R$. Hence $I_e=\m$ for each $e\in\Z_{>0}$ and therefore $\lambda(R/I_e)=\frac{a_e(R)}{p^{e\gamma}}=1$ for each $e\in \Z_{>0}$.

Given finitely generated module $M$ we let $I_e(M)=\{m\in M\mid R\xrightarrow{\cdot F^e_*m} F^e_*M\mbox{ does not split}\}$. It is known, and easy to prove, that $I_e(M)$ is a submodule of $M$ containing $\m^{[p^e]}M$ and $\lambda(M/I(M))=\frac{a_e(M)}{p^{e\gamma}}$ is an integer.

As $M$ is a homomorphic image of $R^{\oplus n}$ for some integer $n \ge 0$, we see that 
\[
\frac{a_{e}(M)}{p^{e\gamma}} \le \frac{a_{e}(R^{\oplus n})}{p^{e\gamma}} 
=\frac{a_{e}(R)n}{p^{e\gamma}} = n.
\] 
Also observe that for all $e' \ge 0$, we have 
$a_{e+e'}(M) = a_{e}(M) a_{e'}(R) + a_{e'}(M_{e}) 
= a_{e}(M) p^{e'\gamma} + a_{e'}(M_{e})$ and hence
\[
\frac{a_{e+e'}(M)}{p^{(e+e')\gamma}} 
= \frac{a_{e}(M)}{p^{e\gamma}} + \frac{a_{e'}(M_e)}{p^{(e+e')\gamma}}
\ge \frac{a_{e}(M)}{p^{e\gamma}}.
\]

In summary, $\{a_e(M)/p^{e\gamma}\}_{e \ge 0}$ is a non-decreasing
sequence of integer values with an upper bound. So it is eventually
constant. All remaining claims follow immediately.
\end{proof}

Let $(R,\m,k)$ be a local ring, not necessarily of prime characteristic, and $M$ a finitely generated $R$-module. Similar to the above, we define $I(M)=\{m\in M\mid R\xrightarrow{\cdot m}M\mbox{ does not split}\}$. Then $I(M)\subseteq M$ is a submodule of $M$ satisfying $\m M\subseteq I(M)$ and $\lambda(M/I(M))=\frk(M)$. We refer to $I(M)$ as the non-split submodule of $M$. If $(R,\m,k)$ is F-finite then $I(F^e_*M)=F^e_* I_e(M)$. Our next lemma studies the behavior of non-split submodules under $R$-linear maps.

\begin{Lemma}\label{mult-c}
Let $(R, \m, k)$ be a local ring (of any characteristic), let $M$, $N$ and $K$ be finitely generated $R$-modules, $f \in \Hom_R(M,N)$ and $g \in \Hom_R(N,K)$.  Let $I(M)$, $I(N)$ and $I(K)$ be the non-split submodules of $M$, $N$ and $K$ respectively. 
\begin{enumerate}
\item We have $\frk(N) \ge \lambda(M/(g\circ f)^{-1}(I(K)))$.
% \item $g \circ f = c 1_M \implies \frk(N) \ge \lambda(M/(I(M):_M c)) =
%   \frk(M) - \lambda(M/(I(M) + cM))$. 
\item Further assume that $R$ is an F-finite ring of prime
  characteristic $p$, $M=K$ and $g \circ f = c 1_M$ for some $c \in R$.
  Then, for all $e \ge 0$, 
\[
a_e(N) \ge a_e(M) - \lambda(M/(I_e(M) + cM)) p^{e \gamma(\m)}.
\]
\end{enumerate}
\end{Lemma}

\begin{proof}
For $(1)$ first observe that $g(I(N)) \subseteq I(K)$. Else, if there exists $n\in I(N)$ such that $g(n)\not\in I(K)$ then there is $\varphi: K\to R$ such that $\varphi(g(n))=1$ contradicting the assumption $n\in I(N)$. Therefore $g(f(f^{-1}(I(N)))) \subseteq g(I(N)) \subseteq I(K)$. In particular, $f^{-1}(I(N)) \subseteq (g\circ f)^{-1}(I(K))$ and hence
\[
\frk(N) = \lambda(N/I(N))) \ge \lambda(M/f^{-1}(I(N))) 
\ge \lambda(M/(g\circ f)^{-1}(I(K))).
\]

We now prove part $(2)$. Suppose $(R,\m,k)$ is an F-finite ring of prime characteristic $p>0$.  For each $e \ge 0$, the induced maps $F^e_*f$ and $F^e_*g$ satisfy $F^e_*g \circ F^e_*f = (F^e_*c) 1_{F^e_*M}$. So $(F^e_*g\circ F^e_*f)^{-1}(I(F^e_*M)) = (I(F^e_*M):_{F^e_*M}F^e_*c)= F^e_*(I_e(M) :_M c)$. By (1), we see
\begin{align*}
a_e(N) = \frk(F^e_*N) &\ge \lambda(F^e_*M/F^e_*(I_e(M) :_M c))
= \lambda(M/(I_e(M) :_M c)) p^{e \gamma(\m)} \\
&= [\lambda(M/I_e(M)) - \lambda(M/(I(M) + cM))] p^{e \gamma(\m)} \\
&= \lambda(M/I_e(M))p^{e \gamma(\m)} - \lambda(M/(I(M) + cM)) p^{e \gamma(\m)} \\
&= a_e(M) - \lambda(M/(I_e(M) + cM)) p^{e \gamma(\m)}.
\end{align*}
The equation $\lambda(M/(I_e(M):_M c)) = \lambda(M/I_e(M)) -
\lambda(M/(I_e(M) + cM))$ follows since length is additive and there is short exact sequence 
\[
0 \to M/(I_e(M):_M c) \to M/I_e(M) \to M/(I_e(M) + cM) \to 0. \qedhere
\]
\end{proof}

We are now ready to accomplish two tasks simultaneously: proving the existence of the F-splitting ratio of a finitely generated module over a local ring, and a uniform convergence result which extends Theorem~\ref{Main Theorem about global F-splitting ratio} to finitely generated modules.

\begin{Theorem}\label{global r(M)}
Let $R$ be an F-finite ring, $M$ a finitely generated $R$-module, and for each prime ideal $Q\in \fsl(R)$ let $\P(Q)$ be the splitting prime ideal of $R_Q$. Then $r_F(M_Q) = r_F(M_{\P(Q)})r_F(R_Q)$ and $\sr(M_Q) = \sr(M_{\P(Q)})$ for all $Q \in \fsl(R)$. Moreover, there exists a constant $C$ such that for all $Q \in \Spec(R)$
and $e \in \Z_{>0}$, 
\[
\left|a_e(M_Q)-p^{e\sr(M_Q)}r_F(M_Q)\right| \leq Cp^{e(\sr(M_Q)-1)}.
\]
\end{Theorem}

\begin{proof}
If $Q\not\in \fsl(R)$ then $a_e(M_Q)=a_e(R_Q)=0$ for all $e\in \Z_{>0}$ and any choice of constant $C\geq 0$ satisfies the desired inequality for all such prime ideals. Furthermore, the F-pure locus $\fsl(R)$ is open in $\Spec(R)$ and is covered by finitely many principal open sets of the form $\Spec(R_f)$ with each $R_f$ being F-pure. Thus we may prove the theorem for each of these pieces of the affine cover and assume for the remainder of the proof that $R$ is an F-pure ring. In particular, $R$ has only finitely many centers of F-purity by Theorem~\ref{Finitely many centers of F-purity}.

 For each center of F-purity $\P$, let $\QP = \{Q
\in \Spec(R) \mid \P(Q) = \P\}$. If $Q\in \Spec(R)$ then $\P(Q)=\P$ if and only if $\P R_Q$ is the splitting prime ideal, i.e. the maximal center of F-purity, of $R_Q$. Thus if $\P_1,\ldots, \P_\ell$ are the centers of F-purity containing, but not equal to $\P$, and if $s\in \left(\P_1\cap \cdots \cap\P_\ell\right) \smallsetminus \P$, then $\QP=V(\P)\cap D(s)$. Here, $D(s)$ denotes the open subset of $\Spec(R)$ consisting of those prime ideals that do not contain $s$. Since $\Spec(R)$ can be covered by finitely many sets of this form it is enough to show the existence of a uniform constant $C$ for which the desired inequality holds for each of the primes in the set $\QP$.

If $r_F(M_{\P}) = 0$ then $r_F(M_Q) = 0$ and the conclusion holds
for all $Q \in \QP$. So we assume $r_F(M_{\P}) > 0$ for the rest of proof. 
Let $\gamma = \gamma(\P) = \sr(M_{\P})$.  By Lemma~\ref{m-center}, there exists $e_0$ such that $a_{e_0}(M_{\P})/p^{e_0\gamma} = r_F(M_{\P})$. Let $a =
a_{e_0}(M_{\P})$. Then $F^{e_0}_*M_{\P} \cong R_{\P}^{\oplus a}
\oplus (M_{e_0})_{\P}$ over $R_{\P}$, for some finitely generated
$R$-module $M_{e_0}$. 
Lifting to $R$, we obtain $R$-linear maps 
\[
R^{\oplus a} \to F^{e_0}_*M \to R^{\oplus a} \quad \text{and} \quad 
F^{e_0}_*M \to R^{\oplus a} \oplus M_{e_0} \to F^{e_0}_*M 
\]
such that both compositions are multiplication by some $c \in R\setminus \P$. 
Applying Lemma~\ref{mult-c} to the composition map $R^{\oplus a} \to F^{e_0}_*M \to R^{\oplus a}$, we see that for all $Q \in \QP$ and $e \ge 0$,
\[
a_{e_0+e}(M_Q) \ge a \cdot (a_e(R_Q) - \lambda(R_Q/(I_e(R_Q) + cR_Q)) p^{e\gamma(Q)}).
\]
Therefore
\begin{align*}
\frac{a_{e_0+e}(M_Q)}{p^{(e_0+e)\gamma}} 
&\ge \frac{a (a_e(R_Q)- \lambda(R_Q/(I_e(R_Q) + cR_Q)) q^{\gamma(Q)})}{p^{(e_0+e)\gamma}} \\
&=\frac{a}{p^{e_0\gamma}} \left(\frac{a_e(R_Q)}{p^{e\gamma}} 
-  \frac{\lambda(R_Q/(I_e(R_Q) + cR_Q))}{p^{e\dim(R_Q/\P R_Q)}}\right)\\
&\ge r_F(M_{\P}) \left(\frac{a_e(R_Q)}{p^{e\gamma}} 
-  \frac{\lambda(R_Q/(Q^{[p^e]} + \P + cR)_Q)}{p^{e\dim(R_Q/\P R_Q)}}\right). %\\
%&\ge r_F(M_{\P}) \left(r_F(R_Q)-\frac{*}{p^e}\right) \\
%&\ge r_F(M_{\P})r_F(R_Q) - \frac{*  r_F(M_{\P})}{p^e} \\
%&\ge r_F(M_{\P})r_F(R_Q) - \frac{*}{p^{e_0+e}}.
\end{align*}

By Theorem~\ref{Main Theorem about global F-splitting ratio} there exists a constant $C_1$, independent of $e$ and $Q \in \mathcal{Q}(\P)$, such that $\frac{a_e(R_Q)}{p^{e\gamma}} \geq r_F(R_Q)-\frac{C_1}{p^e}$, where $\gamma=\sr(M_\P)$ as above. This is because, by Lemma \ref{splitting ratio of a finitely generated module}, we have $\sr(M_\P) = \sr(R_\P)$. Moreover, since $\sdim(R_Q) = \dim(R_Q/\P R_Q)$, we have $\sr(R_\P) = \sr(R_Q)$ for all $Q \in \mathcal{Q}(\P)$. Thus, $\gamma= \sr(M_\P) = \sr(R_\P)=\sr(R_Q)$ for all $Q \in \mathcal{Q}(\P)$. By \cite[Proposition~3.3]{Polstra2018}, there exists a constant $C_2$, independent of $e$ and $Q \in \mathcal{Q}(\P)$, such that $\frac{\lambda(R_Q/(Q^{[p^e]} + \P + cR)_Q)}{p^{e\dim(R_Q/\P R_Q)}}\leq \frac{C_2}{p^e}$. Therefore the constant $C= C_1+r_F(M_\P)C_2p^{e_0}$, which is independent of $e$ and $Q \in \mathcal{Q}(\P)$, is such that
\[
\frac{a_{e_0+e}(M_Q)}{p^{(e_0+e)\gamma}}\geq r_F(M_\P)r_F(R_Q)-\frac{C}{p^{e+e_0}}.
\] 

An argument similar to the above, applied to the composition of maps $F^{e_0}_*M \to R^{\oplus a} \oplus M_{e_0} \to F^{e_0}_*M $, will provide the existence of a constant $C'$, independent of $Q\in \mathcal{Q}(\P)$ and $e$, such that 
\[
\frac{a_{e_0+e}(M_Q)}{p^{(e_0+e)\gamma}}\leq r_F(M_\P)r_F(R_Q)+\frac{C'}{p^{e+e_0}}.
\] 
This shows, in particular, that $\frac{a_e(M_Q)}{p^{e\gamma}}$ converges uniformly to $r_F(M_\P)r_F(R_Q) > 0$, and thus $\gamma = \sr(M_\P) = \sr(M_Q)$. All assertions of the theorem now follow.
\end{proof}

\subsection{Lower semi-continuity of F-splitting ratio function and existence of global F-splitting ratio}

Let $R$ be an F-finite ring and $M$ a finitely generated $R$-module. For each $-1\leq \ell \leq \gamma(R)$ we set $W_\ell(M)=\{P\in \Spec(R)\mid \sr(M_P)=\ell\}$. From previous observation, we have that $W_\ell(M)=W_\ell(R)\cap \fsl(M)$ for all $\ell \geq 0$.

%With an F-finite (not necessarily local) ring $R$, it makes sense to
%define $W_{-\infty}(M) = \Spec(R) - \fsl(M)$ and $W_i(M) = \{P \in
%\Spec(R) \mid \sr(M_P) = i\}$. We have seen $W_{-\infty}(M) \supseteq
%W_{-\infty}(R)$. It is routine to show $W_i(M) = \fsl(M)  
%\cap W_i(R)$ for each $i \ge 0$. (If we denote $W_i = W_i(R)$, we
%not only recover the notation $W_i$ but also give meaning to
%$W_{-\infty}$.) 
%With this more general stratification, many of the theorems in this
%section can be stated without having to assume that 
%$R$ is F-pure, and we can talk about uniform convergence and lower
%semi-continuity on $W_i(M)$ for all $i = -\infty, 0, 1, \cdots,
%\gamma(R)$. (We agree that, when $Q \in W_{-\infty}(M)$, the sequence
%defining $r_F(M_Q) = 0$ is the zero sequence. So usually the
%non-trivial part of proof will focus on $\fsl(M) =\cup_{i \ge 0}W_i(M)$.)

\begin{Theorem}\label{lower semicontinuity of F-splitting ratio} Let $R$ be an F-finite and F-pure ring of prime characteristic $p>0$, set $X=\Spec(R)$ and let $M$ be a finitely generated $R$-module. Then there is a finite stratification of $X$ into locally closed quasi-compact subsets such that the restriction of the F-splitting ratio function on each subset is lower semi-continuous. Specifically, $X=\bigcup_{i=-1}^{\gamma(R)}W_i(M)$, $W_i(M)\cap W_j(M)=\emptyset$ whenever $i\not =j$, the sets $W_i(M)$ are locally closed and quasi-compact, and the function $r_F: \Spec(R)\to \R$ mapping $P\mapsto r_F(M_P)$ is lower semi-continuous when restricted to each $W_i(M)$.
\end{Theorem}

\begin{proof} The functions $a_e:\Spec(R)\to \R$ mapping $P\mapsto a_e(M_P)$ are easily checked to be lower semi-continuous, see \cite[Proposition~2.2]{EnescuYao}. The normalized functions $\tilde{a}_e$ mapping $P\mapsto a_e(M_P)/p^{e\sr(M_P)}$ are therefore lower semi-continuous when restricted to each of the subsets $W_\ell(M)$. It follows that the function $r_F$ is lower semi-continuous when restricted to each $W_\ell(M)$ as it is realized as the uniform limit of lower semi-continuous functions by Theorem~\ref{global r(M)}. It is also easy to see that the sets $W_{-1}(M), W_0(M),\ldots, W_{\gamma(R)}(M)$ are disjoint and $X=\bigcup _{i=-1}^{\gamma(R)}W_i(M)$. It remains to show each of the sets $W_\ell(M)$ are locally closed and quasi-compact.

%The existence of an onto linear map $F^e_*M_P\to R_P$ for some $e\in \Z_{>0}$ is an open condition, i.e. $\fsl(M)$ is an open subset of $X$. Therefore $W_{-1}(M)=\Spec(R)-\fsl(M)$ is a closed subset of $\Spec(R)$ and is therefore locally closed and quasi-compact. Furthermore, since $W_\ell(M)=W_\ell(R)\cap \fsl(M)$ it is now enough to check the sets $W_\ell(R)$ are locally closed and quasi-compact for $0\leq \ell \leq \gamma(R)$.

We adopt the convention that $W_i(M) = \emptyset$ if $i<-1$, and we let $\P(Q)$ denote the splitting prime of $R_Q$. For every $Q \in \Spec(R)$, and every $-1 \leq \ell \leq \gamma(R)$, Theorem \ref{global r(M)} shows that $\sr(M_Q) = \sr(M_{\P(Q)})$, and hence $Q \in W_\ell(M)$ if and only if $\P(Q) \in W_{\ell}(M)$.

Let $\{\p_1,\ldots, \p_N\}$ be the finitely many centers of F-purity of $R$ that are contained in $\fsl(M)$. Relabeling if necessary, we may assume $\gamma(R/\p_j)=\ell$ if and only if $1\leq j\leq i$, and $\gamma(R/\p_j)<\ell$ if and only if $i+1 \leq j \leq t$. Observe that
 \[
 \ds W_\ell(M) = \left(\bigcup_{j=1}^i V(\p_j)\right) \smallsetminus \left(\bigcup_{j=i+1}^t V(\p_j)\right) = \left(\bigcup_{j=1}^i V(\p_j)\right) \cap \left(X \smallsetminus \bigcup_{j=i+1}^t V(\p_j)\right),
 \]
hence it is a locally closed set. Finally, note that every locally closed set of $\Spec(R)$, with $R$ Noetherian, is quasi-compact.
\end{proof}

\begin{Corollary}\label{minimum acheived} Let $R$ be an F-finite and F-pure ring and let $M$ be a finitely generated $R$-module. For $\ell \geq 0$, if $W_\ell(M)\not=\emptyset$, then the F-splitting ratio function defined by $r_F: \Spec(R)\to \R$ mapping $P\mapsto r_F(M_P)$ has a nonzero minimum value when restricted to $W_\ell(M)$.
\end{Corollary}

\begin{proof}
The function $r_F$ is lower semi-continuous when restricted to the non-empty quasi-compact set $W_\ell(M)$ and therefore attains a minimum value.
\end{proof}

The F-splitting ratio function is generally not a lower semi-continuous function when viewed as a function on the spectrum of a ring. We provide an example of such a ring, but first we need a lemma.

\begin{Lemma}\label{lemma on how to show F-splitting ratio not lsc} Let $(R,\m,k)$ be an F-finite and F-pure ring satisfying the following:
\begin{enumerate}
\item $R$ is F-pure;
\item $R$ is not strongly F-regular;
\item $R_P$ is strongly F-regular for all $P\not= \m$;
\item $R_P$ is not regular for some $P\not=\m$.
\end{enumerate}
Then the F-splitting ratio function $r_F: \Spec(R)\to \R$ is not lower semi-continuous.
\end{Lemma}

\begin{proof}
For each $e\in \Z_{>0}$ let $I_e=\{r\in R\mid R\xrightarrow{\cdot F^e_*(r)}F^e_*R\mbox{ is not pure}\}$ be the $e$th splitting ideal of $R$ and set $\mathcal{P}=\bigcap_{e\in \Z_{>0}}I_e$. Recall that $\mathcal{P}$ is referred to as the splitting prime of $R$, and since $R$ is assumed to be not strongly F-regular, the closed set $V(\mathcal{P})$ is contained in the non-strongly F-regular locus of $R$. Therefore $\mathcal{P}=\m$ and it is straightforward to check that $a_e(R)=a_e(R/\m)=[k^{1/p^e}:k]$ for all $e$. In particular, $a_e(R)/p^{e\sr(R)}=1$ for all $e$ and therefore $r_F(R)=1$. However, localizing at a prime $P\not=\m$ for which $R_P$ is not regular it follows $R_P$ is strongly F-regular by assumption but not regular and therefore $r_F(R_P)=\s(R_P)<1$ by \cite[Corollary~16]{HunekeLeuschke} and therefore the F-splitting function is not lower semi-continuous.
\end{proof}

\begin{Theorem}\label{Example of non lsc F-splitting ratio} There exist an F-finite ring $R$ for which the F-splitting ratio function is not lower semi-continuous as a function from $\Spec(R)$ to $\R$.
\end{Theorem}

\begin{proof}
By Lemma~\ref{lemma on how to show F-splitting ratio not lsc} it is enough to show the existence of a non-strongly F-regular local F-pure ring which is strongly F-regular on the punctured spectrum but does not have isolated singularity. Let $k$ be a perfect field of prime characteristic $p$ and let $A$ be a non-regular strongly F-regular ring of finite type over $k$, write $A=k[x_1,\ldots,x_n]/I$, with $I \subseteq (x_1,\ldots,x_n)$, and assume $A_{(x_1,\ldots,x_n)}$ is a non-regular local ring. Let $B=A[v]$ and $R=\{f\in B\mid f(0,\ldots,0,0)=f(0,\ldots, 0,1)\}$. Observe that there are inclusions $A\subseteq R\subseteq B$, all of which are strict since $v^2-v\in R \smallsetminus A$ and $v\in B \smallsetminus R$. Moreover, $R\subseteq B$ is a finite and birational extension of domains. To see this notice that $B=R[v]$ and $v$ is integral over $R$ as $v^2-v\in R$ and the rings $R$ and $B$ have common fraction field since $vx_1,x_1\in R$ and $v=\frac{vx_1}{x_1}$. In particular, $B$ is the normalization of $R$ and the conductor ideal $\c=\Ann_R(B/R)$ defines the non-strongly F-regular locus of $R$ since $B$ is the normalization of $R$ and $B$ is strongly F-regular as polynomial extensions of strongly F-regular rings are strongly F-regular.

We now prove $\c$ is a maximal ideal of $R$. Let $\underline{x}=x_1,\ldots ,x_n$. As an ideal of $B$ there are inclusions $(\underline{x},v(v-1))\subseteq \c \subseteq B$. Therefore the conductor ideal must be one of the following ideals of $B$:
\begin{enumerate}
\item $(\underline{x},v(v-1))$;
\item $(\underline{x},v)$;
\item $(\underline{x},v-1)$;
\item $B$.
\end{enumerate}
Neither $v$ nor $v-1$ is an element of $\c$ and therefore $\c=(\underline{x},v(v-1))$. The $k$-algebra inclusion $R/\c \subseteq B/\c$ is strict, $B/\c$ is a $2$-dimensional $k$-vector space, and $R/\c$ is nonzero. Therefore $R/\c\cong k$ and thus $\c$ is a maximal ideal of $R$. 

We claim that the local ring $R_\c$ will satisfy all desired properties. The ring $R_\c$ is not normal and therefore not strongly F-regular. But if $\p\in \Spec(R) \smallsetminus \{\c\}$ then $R_\p$ is isomorphic to its normalization which is known to be strongly F-regular. In particular, if $\q$ is a prime ideal of $B$ and $\q\cap R$ is not the maximal ideal of $R$ then $R_{\q \cap R}\cong B_\q$. Let $P$ be the prime ideal $(\underline{x})$ of $B$. Then $A_{P\cap A}\subseteq B_P$ is a faithfully flat extension with regular fibers. Since $A_{P\cap A}$ is assumed to be non-regular, the local ring $B_P$ is also non-regular by \cite[Theorem~23.7]{Matsumura}. In particular, the ring $R$ does not have isolated singularity since $R_{P\cap R}\cong B_P$.

 It only remains to check that $R_\c$ is F-pure. To do so we further impose on our ring $A$ that $F_*A\cong A\cdot F_*1\oplus M$, where
 \[
 \ds M \subseteq \mbox{Span}_A\{F_*(x_1^{j_1}\cdots x_n^{j_n})\mid 0\leq j_\ell \leq p-1, \mbox{ and }(j_1,\ldots,j_n)\not=(0,\ldots ,0)\}.\footnote{For example we can take $A=k[x_1,x_2,x_3]/(x_1^2-x_2x_3)$.}
 \] 
Let $\varphi: F_*A\to A$ be the projection of $F_*A$ onto the first factor. Since $A\to B=A[v]$ is a polynomial extension there is $A[v]$-linear map $j_0:F_*B\to B$ such that $j_0(F_*v^i)=0$ whenever $1\leq i \leq p-1$, $j_0(F_*1)=1$, and $j_0|_{F_*A}=\varphi$. It is easy to see that $j_0$ is a splitting of the Frobenius map $B \to F_*B$. For each $1\leq i \leq p-1$, let $j_i: F_*B\to B$ be the $B$-linear map obtained by multiplying elements by $F_*v^{p-i}$, and applying $j_0$. That is, we have $j_i(F_*f)=j_0(F_*(v^{p-i}f))$ for all $F_*f \in F_*B$. Observe that $j_i(F_*v^{i})=v$ and $j_i(F_* v^{j})=0$ whenever $j\not= i$ and $1\leq j\leq p-1$.

Let $f\in B$. We claim that there are unique elements $f_0,\ldots, f_{p-1},h\in B$ such that the following hold:
\begin{enumerate}
\item $f=f_0^p+f_1^pv+\cdots +f^p_{p-1}v^{p-1}+h$;
\item $j_i(F_*f)=f_i$ for each $0\leq i\leq p-1$;
\item $F_*h\in \ker(j_i) $ for each $0\leq i\leq p-1$;
\item $h(0,\ldots, 0, v)=0$.
\end{enumerate}

Observe that if such an expression of $f$ exists then it is clearly unique by properties $(1)$ and $(2)$. To show the existence of such a expression let $f_i=j_i(F_*f)$ and $h=f-(f_0^p+f_1^pv+\cdots +f^p_{p-1}v^{p-1})$. Observe that $j_i(F_*(f_0^p+f_1^pv+\cdots +f^p_{p-1}v^{p-1}))=f_i$ and therefore $h\in \ker(j_i)$ for all $0\leq i\leq p-1$ and our choices of $f_0,\ldots, f_{p-1}$ and $h$ are easily seen to satisfy properties $(1), (2),$ and $(3)$. To see that $h(0,\ldots, 0,v)=0$ write $h=a_0+a_1v+\cdots +a_nv^n$ where each $a_i\in A$. Then $j_0(F_*h)=j_0(F_*a_0)+j_0(F_* a_p)v+j_{0}(F_*a_{2p})v^2+\cdots =0$. Hence $\varphi(F_*a_0)=\varphi(F_*a_p)=\varphi(F_*a_{2p})=\cdots=0$ and therefore $a_{0},a_p,a_{2p},\ldots$ are all elements of  $(x_1,\ldots, x_n)A$. A similar analysis of $j_i(F_*h)$, with $1\leq i\leq p-1$ will show $a_i,a_{p+i},a_{2p+i},\ldots$ are elements of $(x_1,\ldots, x_n)A$ and it easily follows that $h(0,\ldots ,0 ,v)=0$.

Let $j=j_0+j_1+\cdots + j_{p-1}$. We will show that $j:F_*B\to B$ is a $B$-linear map such that 
\begin{enumerate}
\item $j(F_*1)=1$ and 
\item $j(F_*R)\subseteq R$.
\end{enumerate}
In particular, we will have shown $R$ is an F-pure ring.  Given $f\in B$ there is a unique decomposition $f=f_0^p+f_1^pv+\cdots f_{p-1}^pv^{p-1}+h$ satisfying $(1)-(4)$ above. Recall that an element $f\in B$ is an element of $R$ if and only if $f(0,\ldots, 0 ,0)=f(0,\ldots, 0 ,1)$. If $f\in R$ then evaluating the decomposition of $f$ at the points $(0,\ldots,0,0)$ and $(0,\ldots, 0,1)$ shows that 
\[
f_0(0,\ldots, 0 ,0)=(f_0+f_1+\cdots+f_{p-1})(0,\ldots, 0,1).
\]
It is straightforward to check that $j(F_*f)=f_0+(f_1+\cdots + f_{p-1})v$. In particular, 
\[
j(F_*f)(0,\ldots,0,0)=f_0(0,\ldots,0,0)
\]
and
\[
 j(F_*f)(0,\ldots,0,1)=(f_0+f_1+\cdots +f_{p-1})(0,\ldots,0,1).
 \]
Therefore $j(F_*f)\in R$ and we have shown $R$ is F-pure which completes the proof.
\end{proof}

Our next goal is to prove existence of the global F-splitting ratio of a finitely generated module. To do so we recall the following:

\begin{Theorem}[{\cite{Stafford,DSPY2}}]\label{Stafford's Theorem} Let $R$ be a Noetherian ring of Krull dimension $d<\infty$ and $M$ a finitely generated $R$-module. If $\frk_{R_P}(M_P) \geq \dim(R/P) + k$ for all $P \in \Spec(R)$, then $\frk_R(M) \geq k$. In particular, $\frk_R(M) \geq \min\{\frk_{R_P}(M_P) \mid P \in \Spec(R)\} - d$.
\end{Theorem}

Suppose that $R$ is a Noetherian ring, $M$ a finitely generated $R$-module, and $R\to M$ an $R$-module homomorphism. Recall that $R\to M$ splits if and only if the induced map $\Hom_R(M,R)\to \Hom_R(R,R)$ is onto. But a homomorphism of finitely generated $R$-modules being onto is a local condition. Therefore if $R$ is F-finite and $R\to F_*R$ the Frobenius map, then $R\to F_*R$ splitting can be checked locally.

\begin{Proposition} \label{proposition globalVSlocal F-pure} Let $R$ be an F-finite ring. Then $R$ is F-pure if and only if $R_P$ is F-pure for all $P \in \Spec(R)$.
\end{Proposition}

The following is the main theorem of this section. It shows the existence of the global F-splitting ratio, and relates it to the F-splitting ratio of the localization at prime ideals.
\begin{Theorem}\label{Global F-splitting ratio exists} Let $R$ be an F-finite ring of prime characteristic $p>0$ and $M$ a finitely generated module. Then
\begin{enumerate}
\item We have $\displaystyle \sr(M)=\min\{\sr(M_P) \mid P \in \Spec(R)\}$.

\item The limit $\ds r_F(M)=\lim\limits_{e\to \infty}\frac{a_e(M)}{p^{e\sr(M)}}$ exists, and it is positive if $\sr(M)\geq 0$.
\item We have $\displaystyle r_F(M)=\min\{r_F(M_P)\mid \sr(M_P)=\sr(M)\}$. In particular, $r_F(M)$ is positive whenever there exists $e\in\Z_{>0}$ and onto $R$-linear map $F^e_*M\to R$.
\item If $\sr(R)=0$, then the sequence $\{a_e(R)\}$ is the constant sequence $\{1\}$. Therefore, we have $r_F(R)=1$. 
\item If $\sr(M)=0$ then the sequence $\{a_e(M)\}$ is a non-decreasing sequence of eventually positive integers bounded from above, hence is eventually the constant sequence $\{r_F(M)\}$.
\end{enumerate}
\end{Theorem}

\begin{proof}

If there exists a prime ideal $P\in \Spec(R)$ such that $a_e(M_P)=0$ for all $e\in \Z_{>0}$, i.e., if $W_{-1}(M)\not=\emptyset$, then all statements of the theorem trivially follow, and we have $r_F(M)=0$.

%The remainder of the proof will be separated into two cases. The first case will be that $\min\{\sr(M_P)\mid P\in\Spec(R)\}>0$ and the second case will be under the assumption $\min\{\sr(M_P)\mid P\in\Spec(R)\}=0$.

For the remainder of the proof, we assume that $W_{-1}(M) = \emptyset$. Since $a_e(M)\leq a_e(M_P)$ for all $P\in\Spec(R)$, it easily follows that $\sr(M) \leq \min\{\sr(M_P)\mid P\in \Spec(R)\}$.

First, assume that $\min\{\sr(M_P) \mid P \in \Spec(R)\}>0$. For each $e\in \Z_{>0}$, we let $P_e\in \Spec(R)$ be such that $a_e(M_{P_e})=\min\{a_e(M_P)\mid P\in \Spec(R)\}$. If we set $d=\dim (R)$, it follows from Theorem~\ref{Stafford's Theorem} that $a_e(M)\geq a_e(M_{P_e})-d$. Let $C$ be as in Theorem~\ref{global r(M)}, and let $r=\min\{r_F(M_P)\mid P\in \Spec(R)\}$. Such an $r$ exists, and is positive by Corollary~\ref{minimum acheived}. In particular, we have
\begin{align*}
a_e(M)\geq a_e(M_{P_e})-d &\geq r_F(M_{P_e})p^{e\sr(M_{P_e})}-Cp^{e(\sr(M_{P_e})-1)}-d \\ &\geq rp^{e\sr(M_{P_e})} - Cp^{e(\sr(M_{P_e})-1)}-d.
\end{align*}
Since $\sr(M_{P_e})>0$, it follows that $a_e(M)\geq \frac{rp^e}{2}$ for all $e \gg 0$, and therefore $\sr(M)>0$. Moreover, 
%\[
%1\geq \frac{a_e(R_{P_e})}{p^{e\sr(R_{P_e})}}\geq \frac{a_e(R)}{p^{e\sr(R_{P_e})}}\geq r - \frac{C}{p^{e}}-\frac{d}{p^{e\sr(R_{P_e})}}.
%\]
%\[
%\ds \frac{a_e(M)}{p^{e\sr(M_{P_e})}} \geq r_F(M_{P_e}) - \frac{C}{p^e} - \frac{d}{p^{e\sr(M_{P_e})}} \geq r-\frac{C}{p^e} - \frac{d}{p^{e\sr(M_{P_e})}}.
%\]
we have that $\sr(M_{P_e})>\sr(M)$ only for finitely many values of $e$. Else, from the inequalities above we would get
\[
\ds r \leq \frac{a_e(M)}{p^{e\sr(M_{P_e})}} + \frac{C}{p^e} + \frac{d}{p^{e\sr(M_{P_e})}} \leq \frac{a_e(M)}{p^{e(\sr(M)+1)}} + \frac{C}{p^e} + \frac{d}{p^{e(\sr(M)+1)}},
\]
for infinitely many values of $e$. Because $\sr(M) > 0$, the expression on the right hand side can be made arbitrarily close to $0$ for $e \gg 0$, contradicting the fact that $r>0$. Therefore we have $\sr(M_{P_e})=\sr(M)$ for all $e \gg 0$ and, in particular, this gives the reverse inequality $\sr(M) \geq \min\{\sr(M_P)\mid P\in \Spec(R)\}$. This finishes the proof of $(1)$ under the assumption that $\min\{\sr(M_P)\mid P\in\Spec(R)\}>0$.

Continue to assume that $\ell=\sr(M)>0$, and let $P_e\in \Spec(R)$ be as above.  We have already observed that $\sr(M_{P_e})=\ell$ for all $e \gg 0$. Moreover, there are inequalities
\[
\frac{a_e(M_{P_e})-d}{p^{e\ell}}\leq \frac{a_e(M)}{p^{e\ell}}\leq \frac{a_e(M_{P_e})}{p^{e\ell}}.
\]
Under the assumption that $\ell>0$, parts $(2)$ and $(3)$ follow if $\lim\limits_{e\to \infty }\frac{a_e(M_{P_e})}{p^{e\ell}}$ exists and is equal to $\min\{r_F(M_{P})\mid \sr(M_P)=\ell\}$. But this is indeed the case since the F-splitting ratio function restricted to the quasi-compact set $W_\ell(M)=\{P\in\Spec(R)\mid \sr(M_P)=\ell\}$ is the uniform limit of the lower semi-continuous functions $\frac{a_e(-)}{p^{e\ell}}$. In particular, the minimum the functions $\frac{a_e(-)}{p^{e\ell}}$ on $W_\ell(M)$ converges to the minimum of the F-splitting ratio functions on $W_\ell(M)$, using an argument completely analogous to the one used in the proofs of Proposition~\ref{uniform limit chi} and Theorem~\ref{global betas}. This proves $(2)$ and $(3)$ under the assumption that $\min\{\sr(M_P)\mid P\in\Spec(R)\}>0$.

Now we prove $(4)$, so we assume $\sr(R)=0$. By what we have shown above, we must necessarily have $0 = \sr(R) \leq \min\{\sr(R_P) \mid P \in \Spec(R)\} \leq 0$, and thus $\sr(R_P) = 0$ for some $P \in \Spec(R)$. Observe that we have $\sr(R_P) \geq \alpha(P)$, with equality if and only if $PR_P$ is the splitting prime of $R_P$. Thus, $\sr(R_P) = 0$ implies that $PR_P$ is the splitting prime of $R_P$, and that $\kappa(P)$ is perfect. It follows from Lemma~\ref{m-center} and Proposition~\ref{proposition globalVSlocal F-pure} that $1\leq a_e(R)\leq a_e(R_P)=1$ for all $e\in\Z_{>0}$, and therefore $a_e(R)=1$ for all $e\in \Z_{>0}$. This proves $(4)$. 

Now suppose that $\min\{\sr(M_P)\mid P\in\Spec(R)\}=0$. Let $P\in \Spec(R)$ be such that $\sr(M_P)=0$. By Lemma~\ref{splitting ratio of a finitely generated module}, this also gives $\sr(R_P)=\sr(R)=0$. To prove $(5)$, we choose for each $e\in \Z_{>0}$ a direct sum decomposition $F^e_*M\cong R^{\oplus a_e(M)}\oplus M_e$. Then 
\[
F^{e+1}_*M\cong F_*R^{\oplus a_e(M)}\oplus F_*M_e.
\]
As $R$ is F-pure, $F^e_*R^{\oplus a_e(M)}$ has a free summand of rank $a_e(M)$, and therefore $a_{e+1}(M)\geq a_e(M)$. To see that the sequence $\{a_e(M)\}$ is bounded from above, choose an onto map $R^{\oplus N} \to M$. By part $(4)$, the condition that $\sr(R) =0$ implies that $a_e(R^{\oplus N})=N$ for each $e\in \Z_{>0}$, and therefore $a_e(M)\leq N$ for all $e\in \Z_{>0}$. We have now proven, under the assumption that $\min\{\sr(M_P)\mid P\in\Spec(R)\}=0$, that the sequence $\{a_e(M)\}$ is a non-decreasing sequence of non-negative integers, and is therefore eventually the constant sequence $\{r_F(M)\}$. 

To complete the proof it is enough to show that $r_F(M)=\min\{r_F(M_{P})\mid \sr(M_{P})=0\}$, which concludes $(5)$. Moreover, since $\min\{r_F(M_{P})\mid \sr(M_{P})=0\}>0$ by Corollary~\ref{minimum acheived}, this also implies $a_e(M) = r_F(M)>0$ for $e \gg 0$. Hence $\sr(M)=0$, which concludes the proof of parts $(1)$, $(2)$, and $(3)$.

Let $\{P_1,\ldots ,P_s\}$ be the set of maximal objects, with respect to containment, of the set of all centers of F-purity of $R$. We refer to them as the maximal centers of F-purity of $R$. We may assume that $\sr(M_{P_i})=0$ for all $1 \leq i \leq r$, and $\sr(M_{P_i}) > 0$ for $r+1 \leq i \leq s$. From what shown above, we know that for all $e\gg0$ we have $F^e_*M\cong R^{\oplus r_F(M)}\oplus M_e$, where $F^{e'}_*M_e$ does not have a free summand for all $e'\geq 0$.
%$F^e_*M_{P_i}\cong R^{\oplus r_F(M_{P_i})}\oplus M_{i,e}$ for $1\leq i\leq r$ and $F^{e'}_*M_{i,e}$ does not have a free $R_{P_i}$-summand for all $e'\geq 0$.
We claim that $\frk((M_e)_{P_i})\geq 1$ for all $r+1\leq i \leq s$. To see this, 
%The first two items both follow by the observations that the sequences $\{a_e(M)\}$ and $\{a_e(M_{P_i})\}$ are eventually the constant sequences $\{r_F(M)\}$ and $\{r_{F}(M_{P_i})\}$ respectively.
we assume by contradiction that, for some $r+1\leq i\leq s$, we have $\frk((M_e)_{P_i})=0 $ for infinitely many $e \in \Z_{>0}$. Then the splitting rate of $M_{P_i}$ would be $0$, and this contradicts our arrangement of the maximal centers of F-purity of $R$. 

Suppose $r_F(M)<\min \{r_F(M_P)\mid \sr(M_P)=0\}$. Then $\frk((M_{e})_{P_i})>0$ for each $1\leq i \leq r$, and $e \gg 0$. Then for each $1\leq i\leq s$ we can find $m_i\in M_e$ and $h_i\in \Hom_R(M_e, R)$ such that $h_i(m_i)\not \in P_i$. By prime avoidance we can find for each $1\leq i\leq s$ an element $r_i\in \left(\bigcap_{j\not = i}P_j\right) \smallsetminus P_i$. Let $m=\sum r_i m_i$ and $h=\sum r_i h_i$. Then $x:= h(m)=\sum \sum r_ir_jh_i(m_j)\not\in \bigcup _{i=1}^s P_i$. Therefore the element $x$ avoids all maximal centers of F-purity of $R$, hence all centers of F-purity of $R$. In particular, if $Q\in \Spec(R)$, then there exists $e_Q\in \Z_{>0}$ such that $R_Q\xrightarrow{\cdot F^{e'}_* x}F^{e'}_*R_Q$ splits for all $e'\geq e_Q$. Therefore, the union of the sets $U_{e'}:=\{Q\in \Spec(R)\mid R_Q\xrightarrow{\cdot F^{e'}_* x}F^{e'}_*R_Q\mbox{ splits}\}$ is equal to $\Spec(R)$. Moreover, they are open sets, and they form an ascending chain \cite{HochsterHuneke}. By quasi-compactness of $\Spec(R)$, there exists $e'\in\Z_{>0}$ such that $U_{e'}=\Spec(R)$. Therefore $R\xrightarrow{F^{e'}_*x}F^{e'}_*R$ splits, since splitting of is a local condition. Suppose $\varphi:F^{e'}_*R\to R$ satisfies $\varphi(F^{e'}_*x)=1$. Then, the composition $F^{e'}_*M_e\xrightarrow{F^{e'}_*h} F^{e'}_*R\xrightarrow {\varphi }R$ maps $F^e_*h\mapsto 1$, and this contradicts the property that $F^{e'}_*M_e$ does not have a free $R$-summand for all $e'\geq 0$. This completes the proof.
\end{proof}

As in Proposition \ref{prop graded betti} for Frobenius Betti numbers, we show that the F-splitting ratio of a positively graded algebra is equal to the F-splitting ratio at the irrelevant maximal ideal.
\begin{Proposition} \label{prop graded r_F} Let $(R_0,\m_0,k)$ be an F-finite local ring and let $R$ be a positively graded algebra of finite type over $R_0$. Let $R_{>0}$ be the ideal of $R$ generated by elements of positive degree and $\m=\m_0+R_{>0}$. Suppose that $M$ is a finitely generated graded $R$-module. We have the equality $a_e(M) = a_e(M_\m)$. In particular, we have $\sr(M) = \sr(M_\m)$, and $r_F(M) = r_F(M_\m)$.
\end{Proposition}
\begin{proof}
Since $a_e(M) \leq a_e(M_\m)$ always holds, it is sufficient to prove the other inequality. To this end, we observe that $F^e_*M$ is a $\Q$-graded module. Hence, we can find a graded isomorphism $F^e_*M \cong \bigoplus_{i=1}^{b_e} R[n_i] \oplus M_e$, where $n_i \in \Q$, and $M_e$ is a $\Q$-graded module with no graded free summands. Here, $R[n_i]$ denotes the cyclic $\Q$-graded free module whose generator is in degree $-n_i$. We claim that $(M_e)_\m$ has no free summands either. In fact, if it did, there would be a surjective $R_\m$-linear map $(M_e)_\m \to R_\m$. Such a map lifts to an $R$-linear map $\varphi: M_e \to R$ with $\varphi(M_e) \not\subseteq \m$. Since $\Hom_R(M_e,R)$ is a graded module, we can find a graded component $\psi$ of $\varphi$ that still satisfies $\psi(M_e) \not\subseteq \m$. Such a map $\psi$ gives rise to a graded free summand of $M_e$, contradicting our assumptions. This shows that $a_e(M) \geq b_e = a_e(M_\m)$, as claimed.
\end{proof}
\begin{Corollary} Let $R$ and $\m$ be as in Proposition~\ref{prop graded r_F}. % and assume that $Z_R=\Spec(R)$ (for instance, $R$ is a domain). 
We have $\s(R) = \s(R_\m)$ .
\end{Corollary}
\begin{proof}
In our assumptions, the ideal defining the non-strongly F-regular locus is homogeneous \cite[Lemma 4.2]{LyubeznikSmith}. If $R$ is not strongly F-regular, then $R_\m$ is also not strongly F-regular; thus, $\s(R) = \s(R_\m) = 0$ in this case. Now assume $R$ is strongly F-regular. Then $R_\m$ is also strongly F-regular, and thus $\sr(R_\m) = \gamma(R_\m) = \gamma(R)$. %Since $Z_R = \Spec(R)$, we conclude that $\sr(R) = \sr(R_\m) = \gamma(R_\m) = \gamma(R)$, 
Using Proposition~\ref{prop graded r_F}, we conclude that $\sr(R) = \gamma(R)$, and hence $\s(R) = r_F(R) = r_F(R_\m) = \s(R_\m)$. %, again using Proposition \ref{prop graded r_F} for the middle equality.
\end{proof}

\section{Positivity of F-signature of Cartier algebras and strong F-regularity}

This section is devoted to giving a positive answer to \cite[Question 4.24]{DSPY}. We recall the following condition from \cite{DSPY}. For unexplained notation and terminology we refer to Subection 2.4 of the same article.
\begin{Condition}
We say that $(R,\D)$ satisfies condition $(\dagger)$ if at least one of the following
conditions is satisfied:
\begin{itemize}
\item $\D$ satisfies condition $(*)$, as in \ref{condition *}.
\item $\D=\mathcal{C}^{\a^t}$ for some ideal $\a \subseteq R$ and $t>0$.
\item $R$ is normal and $\D=\mathcal{C}^{(R,\Delta)}$ for some effective $\Q$-divisor $\Delta$.
\end{itemize}
\end{Condition}
Using the same notation as in Section \ref{Section global r_F}, we now recall the definition of global F-signature of a pair $(R,\D)$. Given an F-finite and F-pure ring $R$, and a Cartier algebra $\D$, the F-signature of $(R,\D)$ is
\[
\ds \s(R,\D) = \lim_{e \to \infty} \frac{a_e(R,\D)}{p^{e\gamma(R)}}.
\]
When $\D$ is the full Cartier algebra, we simply write $\s(R)$ for $\s(R,\D)$. In this case, if we also have $\gamma(R) = \sr(R)$, the global F-signature $\s(R)$ coincides with the global F-splitting ratio $r_F(R)$ defined in Section \ref{Section global r_F}. The limit above was shown to exist in \cite[Theorem 4.19]{DSPY}. In the same article, a global version of a result of Blickle, Schwede and Tucker \cite{BST2012}, relating the positivity of $\s(R,\D)$ to the strong F-regularity of the pair $(R,\D)$ was established in this setup.
\begin{Theorem}\cite[Corollary 4.23]{DSPY} \label{signature positive with condition} Let $R$ be an F-finite domain, and let $\D$ be a Cartier algebra satisfying condition $(\dagger)$. Then $\s(R,\D)>0$ if and only if $(R,\D)$ is strongly F-regular
\end{Theorem}

The way Theorem \ref{signature positive with condition} was proved in \cite{DSPY} was by exploiting the relation 
\[
\ds \s(R,\D) = \min\{\s(R_P,\D_P) \mid P \in \Spec(R)\}.
\]
Since the strong F-regularity of $(R,\D)$ is equivalent to such minimum being positive, this was sufficient. However, the proof of the equality between the global F-signature of $(R,\D)$ and the minimum of the local invariants required some semi-continuity results, that are only known to hold under the additional assumption that $(\dagger)$ holds \cite{Polstra2018,PolstraTucker}. The goal of this section is to show that Theorem \ref{signature positive with condition} is true without assuming $(\dagger)$. In particular, we will provide a direct way to show that the signature of a strongly F-regular pair $(R,\D)$ is positive, without looking at the corresponding invariants in the localizations at prime ideals.

We start with two preparatory lemmas.
\begin{Lemma}\label{Premultiplication Lemma}{\cite[Lemma~3.13c]{BST2012} and \cite[Lemma~4.2]{PolstraTucker}}  Let $R$ be an F-finite normal domain and let $\varphi\in \Hom_R(F^e_*R,R)$. There exists $0\not= z\in R$ such that for all $n\in \Z_{>0}$, and all $\psi\in \Hom_R(F^{ne}_*R,R)$, there exists $r\in R$ such that
\[
z\psi= \varphi^n(F^{ne}_*r-)
\]
where $\varphi^n= \varphi \circ F^e_*\varphi\circ F^{2e}_*\varphi \circ \cdots \circ F^{(n-1)e}_*\varphi$ and $\varphi^n(F^e_*r-)$ is composition of the maps
\[
F^{ne}_*R\xrightarrow{\cdot F^{ne}_*r}F^{ne}_*R\xrightarrow{\varphi^n}R.
\]
\end{Lemma}

\begin{Lemma}\label{Easy Lemma} Let $R$ be a strongly F-regular F-finite domain. Then there exists $\varepsilon>0$ such that for all $e\in\Z_{>0}$, $a_e(R)\geq \varepsilon \rank(F^e_*R)$.
\end{Lemma}

\begin{proof} As $R$ is strongly F-regular, $\s(R)>0$ by \cite[Theorem~4.16]{DSPY}. Hence, there exists $e'\in\Z_{>0}$ such that for all $e> e'$, $a_e(R)/\rank(F^e_*R)\geq \s(R)/2$. Let 
\[\varepsilon=\min\left\{\frac{a_1(R)}{\rank(F_*R)},..., \frac{a_{e'}(R)}{\rank(F^{e'}_*R)}, \frac{\s(R)}{2}\right\}.
\]
Then $a_e(R)\geq \varepsilon \rank(F^e_*R)$ for all $e\in\Z_{>0}$.
\end{proof}

%$\Delta_{z\psi}=\Delta_\psi+\frac{1}{p^{ne}-1}\div(r^{p^{ne}})$

The following theorem extends \cite[Theorem 2.24]{DSPY}, giving a positive answer to \cite[Question 4.24]{DSPY}.

\begin{Theorem}\label{A} Let $R$ be an F-finite domain and let $\D$ be a Cartier algebra. Then $(R,\D)$ is strongly F-regular if and only if $\s(R,\D)>0$.
\end{Theorem}

\begin{proof} If $(R,\D)$ is not strongly F-regular, then there exists $P\in\Spec(R)$ such that $(R_P,\D_P)$ is not strongly F-regular. Since $a_e(R,\D)\leq a_e(R_P,\D_P)$, we get $\s(R,\D)\leq \s(R_P,\D_P)=0$.

Conversely, suppose that $(R,\D)$ is strongly F-regular. Then $R$ is strongly F-regular and by Lemma~\ref{Easy Lemma} there exists $\varepsilon>0$ such that $a_e(R)\geq \varepsilon\rank(F^e_*R)$ for all $e\in \Z_{>0}$. Let $e_0\in\Z_{>0}$ be such that $\varepsilon\geq \frac{1}{p^{e_0}}$. If $\rank(F^e_*R)=1$ for each $e\in\Z_{>0}$, then $R$ is a perfect field and there is nothing to prove. We assume $R$ is not a perfect field so that, for all $e\geq e_0$, $p^{e_0}$ divides $\rank(F^e_*R)$. Let $\ell_e=\rank(F^e_*R)/p^{e_0}$, so that $a_e(R)\geq \ell_e$ for each $e\in\Z_{>0}$.

Let $e_1>0$ be such that $a_{e_1}(R,\D)>0$, and let $\varphi\in\D_{e_1}$ be a non-zero map. Let $z$ be as in Lemma~\ref{Premultiplication Lemma}. In particular, for each $n\in\Z_{>0}$ and for each $\psi\in\Hom_R(F^{ne_1}_*R, R)$, the map $z\psi$ belongs to $\D_{ne_1}$. Consider integers of the form $e=ne_1\geq e_0$. As $a_e(R)\geq \ell_e$, we can write $F^e_*R\cong R^{\oplus \ell_e}\oplus M_e$ for some $R$-module $M_e$. Let $\lambda_1,\ldots,\lambda_{\ell_e}\in F^e_*R$ form a basis for the free summand $R^{\oplus \ell_e}$ of $F^e_*R$. Denote by $\tilde{\lambda}_i:F^e_*R \rightarrow R$ the $R$-linear map defined by $\lambda_i\mapsto 1$, $\lambda_j\mapsto 0$ for all $j\not= i$, and $x\mapsto 0$ for all $x\in M_e$. 

We chose $0 \not=z\in R$ such that $z\tilde{\lambda}_i\in \D_{e}$, and $z\tilde{\lambda}_i$ maps $\lambda_i\mapsto z$ and $\lambda_j\mapsto 0$ for all $j\not=i$. As $(R,\D)$ is strongly F-regular, there exists $e_2\in\Z_{>0}$ and $\gamma\in \D_{e_2}$ such that $\gamma(F^{e_2}_*z)=1$. Then the $R$-linear maps $\gamma_i:=\gamma\circ F^{e_2}_*z\tilde{\lambda}_i: F^{e+e_2}_*R\rightarrow R$ are elements of $\D_{e+e_2}$ such that $F^{e_2}_*\lambda_i\mapsto 1$ and $F^{e_2}_*\lambda_j\mapsto 0$ for all $j\not=i$. Therefore, for each $e=ne_1\geq e_0$, we have
\[
a_{ne_1+e_2}(R,\D)\geq \ell_{ne_1}= \frac{\rank(F^{ne_1}_*R)}{p^{e_0}}=\frac{\rank(F^{ne_1+e_2}_*R)}{p^{e_0}\rank(F^{e_2}_*R)},
\]
and thus
\[
\s(R,\D)=\lim_{e'\in \Gamma_\D\rightarrow \infty}\frac{a_{e'}(R,\D)}{\rank(F^{e'}_*R)}=\lim_{n\rightarrow \infty}\frac{a_{ne_1+e_2}(R,\D)}{\rank(F^{ne_1+e_2}_*R)}\geq \frac{1}{p^{e_0}\rank(F^{e_2}_*R)}>0. \qedhere
\]
\end{proof}

\begin{Remark} As pointed out above, the proof of Theorem \ref{signature positive with condition} contained in \cite{DSPY} requires the extra assumption that $(\dagger)$ holds, because it is based on the equality $\s(R,\D) = \min\{\s(R_P,\D_P) \mid P \in \Spec(R)\}$. Theorem~\ref{A} settles the positivity of $\s(R,\D)$ for strongly F-regular pairs $(R,\D)$, but it does not indicate any progress in the direction of showing that $\s(R,\D)$ is equal to the minimum of the local invariants. In particular, it does not show the existence of a prime $P\in \Spec(R)$ such that $\s(R,\D)=\s(R_P,\D_P)$. 
\end{Remark}

\bibliographystyle{alpha}
\bibliography{References}

\end{document}